\documentclass[11pt,reqno]{article}
\usepackage{amsmath}
\usepackage{amssymb}
\usepackage{amsthm}
\usepackage{algpseudocode}
\usepackage{authblk}

\usepackage{algorithm}
\usepackage{dsfont}
\usepackage[english]{babel}
\usepackage[dvipsnames]{xcolor}
\usepackage{hyperref}
\usepackage{thm-restate}
\usepackage{tikz}
\usepackage{bbold}
\usepackage{tabularx}
\newcolumntype{C}{>{\centering\arraybackslash}X}
\usepackage{graphicx}
\graphicspath{ {./images/} }
\usepackage{microtype}
\usetikzlibrary{patterns}
\usetikzlibrary{decorations.pathreplacing}
\newcommand{\vertex}[1]{\fill (#1) circle (2.2pt);}
\pgfmathsetmacro{\h}{1}

\usepackage{mathtools}
\mathtoolsset{showonlyrefs} 
\usepackage{zref-clever} 
\zcsetup{nameinlink=false} 

\newtheorem{assumption}{Assumption}
\newtheorem{theorem}{Theorem}[section]
\newtheorem{lemma}[theorem]{Lemma}

\newtheorem*{theorem*}{Theorem}



\usepackage[margin=1in]{geometry}
\usepackage{setspace}
\newcommand{\Z}{\ensuremath{\mathds Z}}
\newcommand{\Pro}{\ensuremath{\mathds{P}}}
\newcommand{\E}{\ensuremath{\mathds{E}}}

\newcommand{\cA}{\mathcal A}
\newcommand{\cC}{\mathcal{C}}
\newcommand{\cH}{\mathcal H}
\newcommand{\cF}{\mathcal{F}}

\newcommand{\cP}{\mathcal{P}}
\newcommand{\cI}{\mathcal{I}}
\newcommand{\cB}{\mathcal{B}}
\newcommand{\cD}{\mathcal{D}}
\newcommand{\cL}{\mathcal{L}}
\newcommand{\barI}{\overline{\cI}}
\newcommand{\cE}{\mathcal E}
\newcommand{\tE}{\widetilde{\mathds E}}
\newcommand{\tnu}{\widetilde{\nu}}
\newcommand{\td}{\widetilde{d}}
\newcommand{\tD}{\widetilde{D}}

\DeclareMathOperator{\Var}{\mathrm{Var}}
\DeclareMathOperator{\Cov}{\mathrm{Cov}}
\newcommand{\lam}{\lambda}
\DeclareMathOperator{\Pois}{\mathrm{Pois}}
\DeclareMathOperator{\Bin}{\mathrm{Bin}}

\newcommand\typ{\mathrm{typ}}
\newcommand\Tri{\mathrm{Tri}}
\newcommand{\dual}[1]{#1^{\star \mathrm{Tri}}}

\newcommand\tilt{\mathrm{tilt}}
\newcommand\unif{\mathrm{unif}}
\newcommand\nufix{\nu_{k,Q,R,S,\typ}^{\unif}}
\newcommand\alg{\mathrm{alg}}
\newcommand\TV{\mathrm{TV}}

\newcommand\ind{\mathrm{ind}}
\newcommand\elem{\mathrm{elem}}

\begin{document}

\title{Local large deviations for triangles in sparse random graphs}

\author{Jade Lintott\thanks{Georgia Institute of Technology, School of Computer Science; \texttt{jlintott3@gatech.edu}} \qquad Will Perkins\thanks{Georgia Institute of Technology, School of Computer Science; \texttt{math@willperkins.org}. Supported in part by NSF grant DMS-2348743.} \qquad Corrine Yap\thanks{Mount Holyoke College, Department of Mathematics and Statistics; \texttt{math@corrineyap.com}}}

\date{\today}

\setlength{\unitlength}{1in}

\renewcommand{\arraystretch}{2}

\maketitle
\begin{abstract}
    We revisit a classic topic in probabilistic combinatorics, the lower-tail large-deviation problem for triangles in the random graph $G(n,p)$.  Here we aim for first-order asymptotics for the quantity $\Pro[X =k] $ with $X$ the number of triangles in $G(n,p)$ and $0 \le k \le (1-\epsilon) \E X$, in the sparse regime in which  the logarithmic asymptotics are Poissonian.  When $k=0$ (the case of triangle-freeness) and  $p = p(n)$ is sufficiently small, first-order asymptotics are known via Janson's inequality and results of Stark and Wormald;  when $k$ is sufficiently close to $\E X$ first-order asymptotics are known via local central limit theorems.  Our main result gives first-order asymptotics for all $k$ in the above range when $p = o(n^{-2/3})$, improving upon the result of Frieze that required $p= o(n^{-4/5})$.

    We also characterize, up to vanishing total variation distance, the distribution of the $k$ triangles in the corresponding conditional distribution and give an efficient algorithm to approximately sample from this conditional distribution.  Notably, when $k = \Theta(\mu)$ there is a transition at $p = \Theta(n^{-4/5})$ from an asymptotically uniform triangle distribution to a distribution asymptotically singular to uniform.
\end{abstract}

\section{Introduction}
\label{section: intro}

What is the probability that the Erd\H{o}s–R\'{e}nyi random graph $G(n,p)$ has exactly (or at most) $k$ triangles?  This basic question has connections to many different areas of probability and combinatorics (including extremal combinatorics, nonlinear large deviation theory, statistical physics, and algorithms) and results on the question have utilized many powerful tools, including Janson's inequality, hypergraph containers, Szemer\'edi's Regularity Lemma and the theory of graphons.

Results about this question can be divided into several different categories, depending on the values of $k$ and $p$, and the accuracy desired in approximating $\Pro_p[X=k]$ or $\Pro_p[X\le k]$ (where we denote by $\Pro_p$  the measure of $G(n,p)$ and $X$  the number of triangles).  There are significant differences in behavior when $p$ is constant (or larger than some critical threshold) and when $p$ is sufficiently small; when $k$ is close to $\E_p[X] = \binom{n}{3} p^3$, much smaller (the lower-tail large deviation problem); and whether one aims for first-order asymptotics of $\Pro_p[X=k]$ or its logarithm.   Our focus in this paper will be on first-order asymptotics in the lower-tail regime when $p$ is sufficiently small.  We begin with some brief background and context.

The case  $k=0$, the probability of having no triangles, is particularly well studied in probabilistic combinatorics, since it is closely related to the problem of enumerating triangle-free graphs of a given density.  Janson’s inequality \cite{PoissAproxRevis} yields
\begin{align}
   \log \Pro_p[X=0] = - (1-o(1)) \E_p[X] ,
\end{align}
for $p=o\left(n^{-1/2}\right)$, while for $p =\omega( n^{-1/2})$, powerful tools from extremal combinatorics (regularity methods~\cite{luczak2000triangle} or hypergraph container methods~\cite{balogh2015independent,saxton2015hypergraph}) give
\begin{align}
\log \Pro_p[X=0] = (1+o(1)) \log \Pro_p[G \text{ bipartite}] \,.
\end{align}

The question of first-order asymptotics of $\Pro_p[X=0]$ is even more delicate. 
When $p=o\left(n^{-4/5}\right)$, Janson's inequality gives $\Pro_p[X=0] = (1+o(1)) \exp ( - \E_p[X])$.
The perturbation method of Wormald \cite{perturbation} gives the first-order asymptotics for $\Pro_p[X = 0]$ when $p = o(n^{-2/3})$, and  Stark and Wormald~\cite{stark_wormald_2018} extend this to $p = o(n^{-1/2-\delta})$ for any fixed constant $\delta>0$. The formula is the exponential of a sum whose number of terms grows as $\delta$ goes to $0$ (see~\cite{BinSub} for a generalization and interpretation of these terms). As an illustrative example, when $p = o(n^{-2/3})$, they prove
\begin{align}
\label{eqSW}
     \Pro_p[X=0]=\left(1+o(1)\right)\exp\left(-\frac{1}{6}n^3p^3+\frac{1}{4}n^4p^5-\frac{7}{12}n^5p^7\right).
\end{align}
For $p$ much larger than $n^{-1/2}$, the picture is qualitatively different, with typical triangle-free graphs being bipartite (or close to bipartite), and the asymptotic formulas reflect this (e.g.~\cite{erdos1976asymptotic,promel1996asymptotic,osthus2003densities,jenssen2025evolution}).  The case of $p= \Theta(n^{-1/2})$ is particularly delicate and only partially understood even on the level of logarithmic asymptotics~\cite{jenssen2024lower,jenssen2026non}. 

Moving beyond the case of triangle-freeness, the asymptotics of the logarithm of the lower tail probability $\Pro_p[X \le k]$ are known in some cases.   When $p = o(n^{-1/2})$, the log asymptotics again follow the Poisson paradigm~\cite{PALD,PoissAproxRevis}, while for $p =\omega(n^{-1/2})$, the log asymptotics are given by the solution of a variational problem~\cite{chatterjee2011large,kozma2023lower}, but this variational problem is only solved in some cases~\cite{zhao_2017}.

For first-order asymptotics of $\Pro_p[X=k]$ when $k >0$, there is the well-studied case when $k$ is close to $\E_p[X]$, in which case local central limit theorems can be proved~\cite{gilmer2014local,sah2022local,araujo2024localcentrallimittheorem}.

However, few results exist for first-order asymptotics of $\Pro_p[X = k]$ in the large-deviation regime when $k > 0$. Frieze \cite{frieze1989small} showed that for $0 \leq k \leq \E_p[X]$ and $p = o(n^{-4/5})$, we have
\begin{align}
    \Pro_p[X = k] = (1+o(1)) \frac{\mu^k e^{-\mu}}{k!},
\end{align}
where $\mu = \E_p[X]$. That is, the asymptotics match the Poisson probability. He posed the question of extending this result to larger $p$, which we address here for $p = o(n^{-2/3})$.

The case $k>0$ is qualitatively different from triangle-freeness: while there is only one way to arrange $0$ triangles on a set of $n$ vertices, there are many non-isomorphic ways to arrange $k$ triangles, and so beyond asking for estimates of $\Pro_p[X=k]$, one can ask for the distribution of the $k$ triangles conditioned on this event. Here we will establish both first-order asymptotics for the probability and a characterization of the conditional distribution of triangles.  In particular we will show that two structural changes occur: for $p$ much smaller than $n^{-4/5}$ (the regime in which Frieze's result applies), the distribution of the $k$ triangles is close to uniform in total variation distance.  For $p$ much larger, but less than $n^{-3/4}$, the distribution is close to uniform with an exponential tilt in the number of shared edges of the $k$ triangles; and between $n^{-3/4}$ and $n^{-2/3}$ the distribution is close to the tilted distribution constrained to form only ``elementary'' connected components of triangles.

\subsection{Main results}

Our first main result gives first-order asymptotics for $\Pro_p[X = k]$.
\begin{theorem} 
\label{main theorem}
Fix $\epsilon > 0$. Let $p=o(n^{-2/3})$, let $k$ be an integer satisfying $0\leq k\leq(1-\epsilon) \E_p[X]$, and let $\theta = \theta(n,p,k) := \frac{6k}{n^3p^3}$ and $\mu = \binom{n}{3}p^3 = \E_p[X]$. Then
\begin{equation}\label{eq:mainThrm}
\Pro_p[X=k] = (1+o(1))\frac{\mu^ke^{-\mu}}{k!}\exp\left(\frac{(1-\theta)^2}{4}n^4p^5 - \frac{(1-\theta)^2(2\theta + 7)}{12}n^5p^7\right)\ .
\end{equation}
\end{theorem}
The parameter $\theta = \frac{k}{\mu} + o(1)$ represents the fraction of the expected triangle count retained in the lower-tail event. When $k = 0$, we have $\theta = 0$ and the formula reduces to the triangle-free asymptotic in \eqref{eqSW}. Moreover, throughout the range $k \leq (1-\epsilon)\mu$, we have $0 \leq \theta \leq 1-\epsilon + o(1)$. Since $p = o(n^{-2/3})$ implies $n^5p^7 = o(n^4p^5)$, the leading correction to the Poisson probability is $\frac{(1-\theta)^2}{4}n^4p^5$. Thus, when $p = o(n^{-4/5})$, the correction exponent is $o(1)$ and we recover the Poisson asymptotic of Frieze, whereas when $p = \Theta(n^{-4/5})$, this correction becomes order one, indicating the presence of this threshold.

To discuss the distribution of the triangles in the graph we need some notation.  Given a graph $G$ on vertex set $V$, let $\Tri(G) := \left\{T \in \binom{V}{3} : G[T] \cong K_3\right\}$ denote the set of triangles of $G$. A \textit{triangle configuration} on $V$ is a set $F \subseteq \binom{V}{3}$ such that if $G_F$ denotes the graph whose edge set is the union of the edges in $F$, then $\Tri(G_F) = F$. 
Let $\cF_k$ be the set of all configurations of $k$ triangles  (on a set of $n$ vertices).

Let $\nu_k^{\unif}$ denote the uniform distribution over $\cF_k$, and let $\nu_{p,k}^{\tilt}$ denote the distribution over $\cF_k$ biased by the number of edges of $G_F$:
\begin{align}
    \nu_{p,k}^{\tilt}(F)\propto \nu_k^{\unif}(F) p^{e(G_F)} \,.
\end{align}
Let $\nu_{p,k}$ denote the distribution of the triangle configuration of $G(n,p)$ conditioned on the event $\{ X =k \}$. 

Recall that the total variation distance between two probability measures $P$ and $P'$ defined on a common sample space $\Omega$ is
$$
    \|P - P'\|_{\TV} := \frac{1}{2}\sum_{x\in \Omega} \left|P(x)-P'(x)\right| \,.
$$
Our next main result states that $\nu_{p,k}$ is close in total variation distance to a conditioned tilted measure. 
Say a configuration is {\em elementary} if every edge-connected component of triangles is one of the following four types: an isolated triangle, two triangles sharing an edge which we call a {\em diamond}, three triangles $T_1, T_2, T_3$ such that $T_1$ and $T_2$ share an edge, $T_2$ and $T_3$ share an edge, and $T_1$ and $T_3$ do not share an edge, which we call a {\em piglet}, and three triangles sharing a common edge which we call a {\em 3-book}. These are depicted in Figure~\ref{fig: side sharing}.

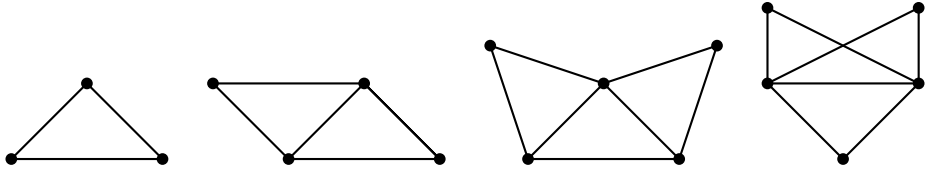
\begin{figure}[h]
\centering
\begin{tikzpicture}[line width=.8pt, line cap=round, line join=round]
    \coordinate (A) at (0,0);
    \coordinate (B) at (2,0);
    \coordinate (C) at (1,1);
    \draw (A)--(B)--(C)--(A);
    \foreach \P in {A,B,C} {\vertex{\P}}
\end{tikzpicture}
\quad
\begin{tikzpicture}[line width=.8pt, line cap=round, line join=round]
  \coordinate (A) at (0,\h);
  \coordinate (B) at (2,\h);
  \coordinate (D) at (1,0);
  \coordinate (E) at (3,0);                    

  \draw (A)--(B);
  \draw (D)--(E);
  \draw (A)--(D)--(B)--(E)--(B);

  \foreach \P in {A,B,D, E} {\vertex{\P}}
\end{tikzpicture}
\quad
\begin{tikzpicture}[line width=.8pt, line cap=round, line join=round]
  \coordinate (A) at (.5,1.5*\h);
  \coordinate (B) at (2,\h);
  \coordinate (C) at (3.5,1.5*\h);
  \coordinate (D) at (1,0);
  \coordinate (E) at (3,0);                     

  \draw (A)--(B)--(C);
  \draw (D)--(E);
  \draw (A)--(D)--(B)--(E)--(C);

  \foreach \P in {A,B,C,D,E} {\vertex{\P}}
\end{tikzpicture}
\quad
\begin{tikzpicture}[line width=.8pt, line cap=round, line join=round]
  \coordinate (A) at (0,\h);
  \coordinate (B) at (2,\h);
  \coordinate (C) at (0,0);
  \coordinate (D) at (2,0);
  \coordinate (E) at (1,-\h);

  \draw (A)--(C);
  \draw (B)--(D);
  \draw (C)--(D);
  \draw (A)--(D);
  \draw (C)--(B);
  \draw (C)--(E);
  \draw (D)--(E);

  \foreach \P in {A,B,C,D,E} {\vertex{\P}}
\end{tikzpicture}
  \caption{The possible triangle components in an elementary configuration: a triangle, a diamond, a piglet, and a 3-book.}
  \label{fig: side sharing}
\end{figure}
Let $\cF_{k, \mathrm{elem}}$ be the set of $k$-triangle configurations which are elementary, and let $$\nu_{p,k, \mathrm{elem}}^{\tilt} = \nu_{p,k}^{\tilt} (\cdot \mid F \in \cF_{k, \mathrm{elem}})\ .$$

\begin{theorem} 
\label{theorem: dist}
Fix $\epsilon > 0$. For $p=o(n^{-2/3})$ and $k$ an integer satisfying $0\leq k\leq(1-\epsilon) \E_p[X]$,
\begin{align}
    \| \nu_{p,k}- \nu_{p,k,\mathrm{elem}}^{\tilt}\|_{\TV}=o(1) \, .
\end{align} 
\end{theorem}
While this theorem characterizes the conditional distribution for the entire range $p = o(n^{-2/3})$, we can be even more specific and show where the exponential tilting becomes significant (around $p=n^{-4/5}$) and when the restriction to elementary components becomes significant (around $p= n^{-3/4}$). 

\begin{theorem}\label{theorem: typical-char}
Under the assumptions of Theorem~\ref{theorem: dist}, 
    \begin{enumerate}
    \item (Closeness to uniform) If $p = o(n^{-4/5})$, then $\|\nu_{p,k}^{\tilt} - \nu_k^{\unif}\|_{\TV} = o(1)$. \label{theorem: typical-char-part1}
    \item (Far from uniform)  If  $n^{-4/5} \ll p \ll n^{-3/4}$ and $k =\Theta(\mu)$, then $\|\nu_{p,k}^{\tilt} - \nu_k^{\unif}\|_{\TV} = 1-o(1)$. \label{theorem: typical-char-part2}
        \item (Closeness to tilted) If $p = o(n^{-3/4})$, then $\|\nu_{p,k, \elem}^{\tilt} - \nu_{p,k}^{\tilt}\|_{\TV} = o(1)$. \label{theorem: typical-char-part3}
        
        \item (Far from tilted) If $p = n^{-\alpha}$ for some $\frac23 < \alpha < \frac34$ and $k= \Theta(\mu)$, then 
        $\|\nu_{p,k, \elem}^{\tilt} - \nu_{p,k}^{\tilt}\|_{\TV} = 1- o(1)$ \label{theorem: typical-char-part4}
    \end{enumerate}

\end{theorem}

\zcref[S]{theorem: dist} suggests a route towards an efficient sampling algorithm for $G(n,p)$ conditioned on $\{ X= k\}$:  sample a configuration of $k$ triangles from $\nu_{p,k,\mathrm{elem}}^{\tilt}$ and then sample the remaining  edges with probability $p$,  conditioning on forming no additional triangles. Using this template and some Markov chain analysis, we provide an efficient sampling algorithm for the conditional distribution.

\begin{theorem}
\label{theorem: sample}
Fix $\epsilon > 0$. For $p=o(n^{-2/3})$ and $k$ an integer satisfying $0\leq k\leq(1-\epsilon) \E_p[X]$,
 let $\pi_{p,k}$ denote the distribution of graphs $G \sim G(n,p)$ conditioned on having exactly $k$ triangles. There is an $O(n^3\log n)$-time algorithm to sample from a distribution $\pi_{\alg}$ such that 
    \begin{align}
        \| \pi_{\alg}-\pi_{p,k}\|_{\TV}=o(1).
    \end{align}
\end{theorem}

\subsection{Proof outline}
\label{secProofOutline}

We begin by decomposing the event $\{X(G)=k\}$ according to the triangle configuration of $G$. We write
\begin{align}
    \Pro_p[X(G)=k]
    =\sum_{F\in\cF_k}\Pro_p[\Tri(G)=F]
    =\sum_{F\in\cF_k} p^{e(G_F)}\cP_{p}(F),
\end{align}
where $\cP_{p}(F)$ is defined to be the probability that, after conditioning on the presence of the edges of $G_F$, the remaining  edges  create no additional triangles  when included independently with probability $p$. This separates the problem into a counting problem for triangle configurations and a subgraph non-existence problem of forbidding new triangles. 

The next step is to show that only a simple subclass of configurations contributes to the first-order asymptotics of the probability. This  allows us to ignore particularly difficult to analyze structures in triangle configurations. Using first-moment bounds together with the FKG inequality under the event $\{X(G)\leq k\}$, we prove that with high probability the triangle configuration avoids  a set of specific forbidden subgraphs.

We then estimate $\cP_{p}(F)$ uniformly over the remaining typical configurations. After the edges of $G_F$ are fixed, a potential additional triangle may use zero, one, or two edges already present in $G_F$, so the event of creating no new triangles can be encoded as the event that a certain non-uniform hypergraph contains no occupied hyperedge. We follow the framework of Mousset, Noever, Panagiotou, and Samotij in \cite{BinSub} to show that the expectation  over typical $F$ of $\cP_{p}(F)$ is 
\begin{align}
   \exp\left(-\frac{1}{6}n^3p^3-3knp^2+\frac{1}{4}n^4p^5+6kn^2p^4-\frac{9k^2p}{n}-\frac{7}{12}n^5p^7  +o(1)\right) .
\end{align}

The remaining task is to evaluate the weighted sum over typical configurations. A configuration with $\zeta$ shared edges receives edge weight $p^{3k-\zeta}$, so shared edges are rewarded by a factor of $p^{-\zeta}$, but configurations with many shared edges are combinatorially rarer. We show that these two effects balance in a narrow window of values of $\zeta$, while configurations outside that window contribute only $o(1)$ to the full sum. Summing over this window contributes the remaining terms to the exponent; after combining them with the above estimate for $\cP_p(F)$, the corrections factor as 
$$\frac{(1-\theta)^2}{4}n^4p^5 - \frac{(1-\theta)^2(2\theta + 7)}{12}n^5p^7$$
yielding~\eqref{eq:mainThrm}.

\subsection{Notation}
For convenience, we restate here the definitions from the introduction and introduce notational conventions. 

Throughout the paper, we assume the conditions of the main results.
\begin{assumption}
    \label{assumption1}
    Fix a constant $\epsilon>0$. Then we take $p$ and $k$ such that
    \begin{itemize}
\item $p=o(n^{-2/3})$
\item $k$ is an integer satisfying  $0 \le k \le (1-\epsilon)\E_p[X(G)]$ \,.
    \end{itemize}
\end{assumption}

For a graph $H$, we write $V(H)$ and $E(H)$ for its vertex and edge sets and $v(H) := |V(H)|$ and $e(H):=|E(H)|$ for its order and size, respectively. We denote by $\deg_H(v)$  the degree of a vertex $v$ in $H$. All asymptotics are taken as $n$ goes to $\infty$. We say {\em with high probability} or {\em whp} for short to mean with probability tending to 1 as $n \to \infty$. We use the standard asymptotic notation $o(1)$, $O(\cdot)$, and $\Omega(\cdot)$ with respect to $n$.  The shorthand $f \ll g$ will occasionally be used to mean $f = o(g)$.

Given a graph $G$ on vertex set $V$, recall the set of triangles of $G$ is denoted $\Tri(G) := \left\{T \in \binom{V}{3} : G[T] \cong K_3\right\}$. For $F \subseteq \binom{V}{3}$, let $G_F$ denote the graph on vertex set $V$ whose edge set is the union of edges covered by triples in $F$. A \textit{triangle configuration} on $V$ is a set $F \subseteq \binom{V}{3}$ such that $\Tri(G_F) = F$. 
We use the shorthand $e_F = e(G_F)$ and $\deg_F(v) = \deg_{G_F}(v)$.  The \emph{triangle-dual graph} of a graph $G$ has one vertex for each triangle and an edge whenever the corresponding triangles share an edge; we denote this $\dual{G}$. We say two triangles are {\em edge-adjacent} if they share exactly one edge (equivalently, they correspond to adjacent vertices in $\dual{G}$), and a set of triangles is {\em edge-connected} if the subgraph of $\dual{G}$ induced by those triangles is connected.

We define $\cF_k$ to be the set of all triangle configurations with $k$ triangles, and we similarly define  $\cF_{\leq k}$ to be the set of all triangle configurations with at most $k$ triangles.

Throughout the remainder of the paper, let $G \sim G(n,p)$ and $V = V(G)$ unless otherwise stated. The symbol $\Pro$ will always refer to the distribution over $G(n,p)$, unless a different distribution is specified, whereas the symbol $\nu$ will be used for distributions on triangle configurations. In particular, 
\begin{itemize}
    \item $\Pro = \Pro_p$ is the random graph distribution $G(n,p)$, 
    \item $\pi_{p,k}$ is the distribution of $G(n,p)$ conditioned on the event $\{X = k\}$, 
    \item $\nu_{p,k}$ is the distribution on triangle configurations induced by $\pi_{p,k}$,  
    \item $\nu_k^{\unif}$ is the uniform distribution on $\cF_k$, 
    \item $\nu_{p,k}^{\tilt}$ is the distribution on $\cF_k$ where each configuration $F$ is given weight proportional to $\nu_k^{\unif}p^{e_F}$.
\end{itemize}

\subsection*{Organization}
The remainder of the paper is organized as follows. In Section~\ref{section: atypical configs}, we identify a class of typical triangle configurations and show that $G(n,p)$ conditioned on having at most $k$ triangles has a typical triangle configuration with high probability. In Section~\ref{section: addit tri prob}, we estimate the probability that a fixed typical triangle configuration extends to a graph with no additional triangles and analyze this estimate after conditioning on the numbers of diamonds, piglets, and 3-books. In Section~\ref{section: overlapping edges}, we show that these component counts are asymptotically distributed as independent Poisson random variables and obtain the asymptotic count of typical configurations. In Section~\ref{section: final formulas}, we combine these estimates to prove Theorem~\ref{main theorem}. In Section~\ref{sec: sampling}, we study the conditional distribution of the triangles and the remaining edges, proving Theorems \ref{theorem: dist}, \ref{theorem: typical-char}, and \ref{theorem: sample}.

\subsubsection*{Note on AI use} The main ideas and arguments all came from the authors. AI (ChatGPT Plus) was used in checking proofs; it suggested simplifications to the formula~\eqref{eq:mainThrm}, the proof of Lemma~\ref{lemma:expectation-exchange}, and the proofs of \zcref[S]{theorem: typical-char}, Parts \ref{theorem: typical-char-part3} and \ref{theorem: typical-char-part4}, all of which were adopted. 

\section{Atypical Configurations}
\label{section: atypical configs}
Towards the proof of \zcref[S]{main theorem}, we first identify certain subgraphs that $G(n,p)$ conditioned on having $k$ triangles will avoid with high probability. We first observe that by the FKG inequality, conditioning on at most $k$ triangles cannot increase the probability that an unlikely subgraph appears in $G$.

\begin{lemma}
\label{lemma: FKG}
    Let $H$ be an unlabeled graph on at most $n$ vertices and let $N_H(G)$ be the number of labeled copies of $H$ in $G$. If $\Pro[N_H(G)>0]=o(1)$ then  $\Pro[N_H(G)>0\mid X(G)\leq k]=o(1)$ .
\end{lemma}
\begin{proof}
    We will prove this using the FKG inequality. Having a copy of $H$ as a subgraph is an increasing property in a graph, while having at most $k$ triangles is a decreasing property. 
    This means by the FKG inequality we have that 
    \begin{align}
        \Pro[N_H(G)>0]&\geq\Pro[N_H(G)>0\mid X(G)\leq k]. \label{eq: conditioned prob}
    \end{align}
    Since $\Pro[N_H(G)>0]=o(1)$, the inequality above implies $\Pro[N_H(G)>0\mid X(G)\le k]=o(1)$.
\end{proof}

We now identify a precise set of subgraphs consisting of components of triangles that will be unlikely to appear in $G$. 
To that end, let $r=\left\lceil \frac{\log n}{6}\right\rceil$, and let the {\em triangle-intersection graph} of a configuration have one vertex for each triangle and an edge whenever the corresponding triangles share a vertex. We define $\mathcal A = \cA_n$ to be the set of isomorphism types of inclusion-minimal graphs having one of the following properties:
\begin{enumerate}
    \item  $K_4$, the clique on 4 vertices; \label{atypicaltype1}
    \item graphs induced by a triangle configuration consisting of exactly four triangles whose triangle-dual graph is connected; \label{atypicaltype2} 
    \item the graph consisting of $r$ triangles that have exactly one common vertex and are otherwise vertex-disjoint; \label{atypicaltype3}
    \item graphs induced by a triangle configuration whose triangle-intersection graph is the $4$-cycle $C_4$.
    \label{atypicaltype4} 
\end{enumerate}
Examples of each of the 4 types are depicted in Figure \ref{fig: forbidden graphs}.
$\mathcal{A}$ will constitute a set of atypical subgraphs that will whp not appear in $G$.

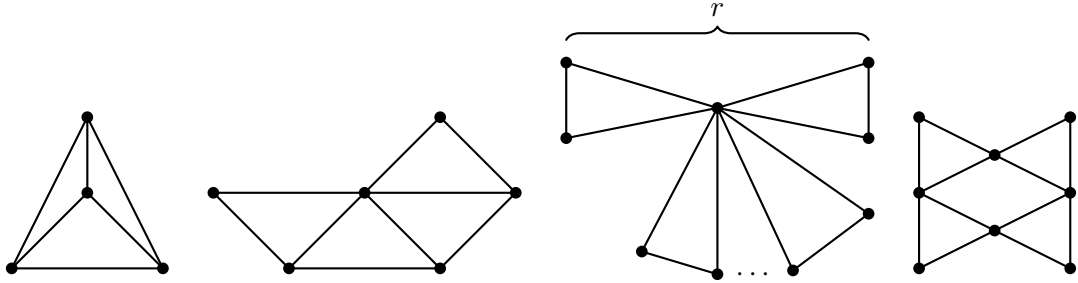
\begin{figure}[h]
\centering
\begin{tikzpicture}[line width=.8pt, line cap=round, line join=round]
  \coordinate (A) at (0,0);
  \coordinate (B) at (1,2*\h);
  \coordinate (C) at (2,0);
  \coordinate (O) at (1,\h);
  \draw (A)--(B)--(C)--cycle;
  \draw (A)--(O)--(B);
  \draw (O)--(C);

  \foreach \P in {A,B,C,O} {\vertex{\P}}
\end{tikzpicture}
\quad
\begin{tikzpicture}[line width=.8pt, line cap=round, line join=round]
  \coordinate (A) at (0,\h);
  \coordinate (B) at (2,\h);
  \coordinate (C) at (4,\h);
  \coordinate (D) at (1,0);
  \coordinate (E) at (3,0);
  \coordinate (F) at (3,2*\h);                      

  \draw (A)--(B)--(C);
  \draw (D)--(E);
  \draw (A)--(D)--(B)--(E)--(C)--(F)--(B);

  \foreach \P in {A,B,C,D,E,F} {\vertex{\P}}
\end{tikzpicture}
\quad
\begin{tikzpicture}[line width=.8pt, line cap=round, line join=round, baseline=(base)]
\coordinate (base) at (0,1.2);
  \coordinate (A) at (1,1.5);
  \coordinate (B) at (2,3.4);
  \coordinate (C) at (4.0,2.0);
  \coordinate (D) at (2,1.2);
  \coordinate (E) at (3,1.25);
\coordinate (F) at (4,3);
\coordinate (G) at (4,4);
\coordinate (H) at (0,3);
\coordinate (I) at (0,4);

  \draw (A)--(B)--(C)--(E);
  \draw (A)--(D)--(B);
  \draw (B)--(E);
  \draw (F)--(B)--(G)--(F);
    \draw (H)--(B)--(I)--(H);
\node at (2.5,1.2) {$\cdots$};
    \draw[decorate, decoration={brace, amplitude = 5pt}] (0, 4.3) -- (4, 4.3) node[midway, above=5pt]{$r$};
  \foreach \P in {A,B,C,D,E, F, G, H, I} {\vertex{\P}}
\end{tikzpicture}
\quad
\begin{tikzpicture}[line width=.8pt, line cap=round, line join=round]
  \coordinate (L) at (0,0);
  \coordinate (M) at (0,\h);
  \coordinate (N) at (0,2*\h);
  \coordinate (P) at (1,.5*\h);
  \coordinate (Q) at (1,1.5*\h);
  \coordinate (R) at (2,0);
  \coordinate (S) at (2, \h);
\coordinate (T) at (2, 2*\h);

  \draw (L)--(M)--(N)--(Q)-- (M);
  \draw (L)-- (P)-- (M);
\draw (R)-- (P)-- (S)--(R);
\draw (S)-- (T)-- (Q)--(S);

  \foreach \P in {L,M,N,P,Q,R,S, T} {\vertex{\P}}
  \end{tikzpicture}
  \caption{Examples of each of the 4 atypical subgraph categories in $\mathcal{A}$.}
  \label{fig: forbidden graphs}
\end{figure}
 
If a component of $\dual{G_F}$ has order 4 or more, then $G_F$ must contain a subgraph of Type \ref{atypicaltype1} or \ref{atypicaltype2}. Thus, for typical triangle configurations $F$, the components of the dual graph $\dual{G_F}$ have order at most 3. The possible components of order 2 and 3 are exactly a diamond, a piglet, and a 3-book, shown in Figure~\ref{fig: side sharing}, which correspond to $P_2, P_3$, and $C_3$ components in $\dual{G}$, respectively.

Given a triangle configuration $F$, let $Q = Q(F)$ be the number of diamonds, $R = R(F)$ be the number of piglets, and $S = S(F)$ be the number of 3-books. 
Recall $\cF_{\leq k}$ is the set of all configurations of at most $k$ triangles, and let $\cF_{\leq k, \typ}$ be the set of all configurations $F\in\cF_{\leq k} $ such that $N_H(G_F)=0$ for all $ H\in \mathcal{A}$. 

Our next lemma shows that whp, the graphs with at most $k$ triangles do not contain any of the atypical triangle configurations.
\begin{lemma}
\label{lemma: forbidden subgraphs}
    Let  $\cF_{\leq k, \typ}$ be as defined above. Then 
    \begin{align}
      \Pro[X(G)\leq k] =(1+o(1)) \Pro[\Tri(G)\in  \cF_{\leq k, \typ}].
    \end{align}
\end{lemma}
\begin{proof}
The claim is equivalent to the statement 
\begin{align}
    \Pro\left[\Tri(G)\in  \cF_{\leq k, \typ} \mid X(G)\leq k\right]=1-o(1).
\end{align}
It is enough to show $\Pro[N_H>0]=o(1)$ for all $H \in \mathcal{A}$; since $\cA$ consists of a bounded number of isomorphism types, we may then apply \zcref[S]{lemma: FKG} and a union bound to conclude. 

For each $H \in \cA$, we will show that $\E_p[N_H]=o(1)$ and then apply Markov's inequality. The first type of subgraph in $\cA$ is $K_4$; the expected number of these is $O(n^4p^6)=o(1)$. 

The second type (\ref{atypicaltype2}) is four edge-connected triangles; let these be labeled $T_1,T_2,T_3, T_4$. We now consider how many edges and vertices these cover. Since the triangle-dual graph is connected, we may order the triangles so that for each $j\in\{2,3,4\}$, the triangle $T_j$ shares an edge with at least one of $T_1,\dots,T_{j-1}$. Thus, $T_1 \cup T_2$ forms a diamond with $4$ vertices and $5$ edges. 
Conditioned on these, if $T_3$ shares distinct edges with $T_1$ and $T_2$, then $H$ must form a $K_4$ which was discussed previously. Else, $T_3$ contains two edges and one vertex not covered by $T_1 \cup T_2$ and similarly $T_4$ contains either two edges and one vertex or one edge and no vertices not covered by the previous triangles. Thus the possible orders and sizes are $(v(H), e(H)) \in \{(5, 8), (6,9)\}$, and so the expected number of each in $G$ is $O(n^5p^8 + n^6p^9) = o(1)$.

The third type (\ref{atypicaltype3}) is a fixed vertex incident to $r = \left\lceil \frac{\log n}{6} \right\rceil$ edge-disjoint triangles; we compute the expected number of copies as at most
\begin{align}
    \frac{n^{2r+1}p^{3r}}{r!}&=o\left(\frac{n}{e^{(1/6 + o(1))\log n \log \log n}}\right)=o(1).
\end{align}
The fourth type (\ref{atypicaltype4}) is a 4-cycle in the triangle-intersection graph of vertex-connected triangles;
one can check that these have four possible structures, shown in Figure \ref{fig: 4 cycles}, or they contain a size 4 edge-connected component.
\begin{figure}[h]
\centering
\begin{tikzpicture}[line width=.8pt, line cap=round, line join=round]
  \coordinate (L) at (0,0);
  \coordinate (M) at (0,\h);
  \coordinate (N) at (0,2*\h);
  \coordinate (P) at (1,.5*\h);
  \coordinate (Q) at (1,1.5*\h);
  \coordinate (R) at (2,0);
  \coordinate (S) at (2, \h);
\coordinate (T) at (2, 2*\h);

  \draw (L)--(M)--(N)--(Q)-- (M);
  \draw (L)-- (P)-- (M);
\draw (R)-- (P)-- (S)--(R);
\draw (S)-- (T)-- (Q)--(S);

  \foreach \P in {L,M,N,P,Q,R,S, T} {\vertex{\P}}
  \end{tikzpicture}
  \begin{tikzpicture}[line width=.8pt, line cap=round, line join=round]
  \coordinate (L) at (0,0);
  \coordinate (M) at (0,\h);
  \coordinate (N) at (1,2*\h);
  \coordinate (P) at (1,.5*\h);
  \coordinate (Q) at (1,1.5*\h);
  \coordinate (R) at (2,0);
  \coordinate (S) at (2, \h);
\coordinate (T) at (1, 2*\h);

  \draw (L)--(M)--(N)--(Q)-- (M);
  \draw (L)-- (P)-- (M);
\draw (R)-- (P)-- (S)--(R);
\draw (S)-- (T)-- (Q)--(S);

  \foreach \P in {L,M,N,P,Q,R,S, T} {\vertex{\P}}
  \end{tikzpicture}
  \begin{tikzpicture}[line width=.8pt, line cap=round, line join=round]
  \coordinate (L) at (1,0);
  \coordinate (M) at (0,\h);
  \coordinate (N) at (1,2*\h);
  \coordinate (P) at (1,.5*\h);
  \coordinate (Q) at (1,1.5*\h);
  \coordinate (R) at (1,0);
  \coordinate (S) at (2, \h);
\coordinate (T) at (1, 2*\h);

  \draw (L)--(M)--(N)--(Q)-- (M);
  \draw (L)-- (P)-- (M);
\draw (R)-- (P)-- (S)--(R);
\draw (S)-- (T)-- (Q)--(S);

  \foreach \P in {L,M,N,P,Q,R,S, T} {\vertex{\P}}
  \end{tikzpicture}
\begin{tikzpicture}[line width=.8pt, line cap=round, line join=round]
  \coordinate (A) at (0,.5*\h);
  \coordinate (B) at (2,-\h);
  \coordinate (C) at (4,.5*\h);
  \coordinate (D) at (1.5,0);
  \coordinate (E) at (2.5,0);
  \coordinate (F) at (2, \h);

  \draw (A)--(B)--(C);
  \draw (D)--(E);
  \draw (A)--(D)--(B)--(E)--(C);
  \draw (A)--(F)--(C)--(A);

  \foreach \P in {A,B,C,D,E,F} {\vertex{\P}}
\end{tikzpicture}
  \caption{The relevant subgraphs that include an induced four cycle.}
  \label{fig: 4 cycles}
\end{figure}
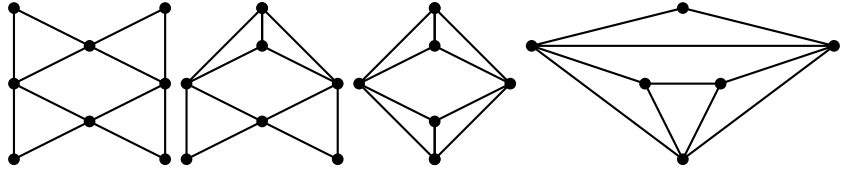
The possible subgraphs satisfy $(v(H), e(H)) \in \{(8,12), (7,11), (6,10)\}$ corresponding to 0, 1, and 2 consecutive pairs sharing an edge. It follows that the expected number of such subgraphs is $o(1)$. 
\end{proof}

Using the previous result, we show that we can more broadly discard all graphs with high-degree vertices. We write $\deg_F(v)$ to refer to the degree of a vertex $v$ in the graph $G_F$.
\newpage
\begin{lemma}\label{lemma:bdd-deg}
   $\deg_F(v) \leq \log n$ for all $F\in \cF_{\leq k, \typ}$ and $v\in V(G_F)$. 
\end{lemma}

\begin{proof}
    Given $v \in V(G_F)$, let $\mathcal M$ be a maximal family of edge-disjoint triangles containing $v$. By definition of Type \ref{atypicaltype3}, we must have $|\mathcal M| \leq r-1$. Moreover, every triangle containing $v$ must be either in $\mathcal M$ or edge-adjacent to a member of $\mathcal M$. Since components of the triangle-dual have order at most 3, each member of $\mathcal M$ is edge-adjacent to at most two other triangles. Thus, $v$ is contained in at most $3(r-1)$ triangles, and hence $\deg_F(v) \leq 6(r-1) \leq \log n$.
\end{proof}

\section{Estimating the Probability of No Additional Triangles}
\label{section: addit tri prob}

In this section, we decompose the probability that $G$ has $k$ triangles into two parts: the induced distribution on triangle configurations and  the distribution on the remaining edges, conditioned on an underlying set of triangles. 

We begin with the latter; for $F \in \cF_{\leq k}$ define
\begin{align}
\label{eq1}
    \cP_p(F) = \cP(F):= \Pro\left[X(G_F\cup G)=X(G_F)\right]\ .
\end{align}
In other words, $\cP(F)$ is defined to be the probability that $G \sim G(n,p)$ does not introduce any new triangles when sampled on top of $G_F$.
Therefore,
\begin{align}
\Pro[X \leq  k] &=\sum_{F \in \cF_{\leq k}} \Pro\left[\Tri(G)=F\right]\\
&=\sum_{F \in \cF_{\leq k}} p^{e_F} \cdot \cP(F)\\
&= (1+o(1)) \sum_{F \in \cF_{\leq k, \typ}} p^{e_F} \cdot \cP(F)
\label{eq:total-prob-sum}
\end{align}
where the last line follows from \zcref[S]{lemma: forbidden subgraphs}.

Letting $\cF_{k,\typ} := \cF_{k}\cap \cF_{\leq k, \typ}$, we approximate the probability of triangle configurations in $\cF_{k,\typ}$ in terms of the typical triangle component types described in the last section.
\begin{lemma}
\label{lemma: formula}
    Uniformly over $F\in \cF_{k, \typ}$, 
    \begin{align}
        \cP(F)=(1+o(1))\exp\bigg(&-\frac{n^3p^3}{6} +3kp+e_F(p-np^2+2n^2p^4)+pQ+2pR+3pS\\
        &+ \frac{n^4p^5}{4} -\frac{7n^5p^7}{12} +\left(-\frac{1}{2}p+np^3\right)\sum_{v\in V} \deg_F(v)^2 \bigg)\ .\label{eq:formula_M}
    \end{align}
\end{lemma}
Note that
\begin{align}
    e_F=3k-Q(F)-2R(F)-2S(F).
\end{align}

We will prove \zcref[S]{lemma: formula} using \cite[Theorem 10]{BinSub}. The setting of \cite{BinSub} is a hypergraph with edges of bounded order; Theorem 10 shows that under certain codegree conditions, the probability that a binomial random subset of the vertices contains any edges is asymptotically equal to a Janson-type expression given by the joint cumulants and joint moments over small sets of edge indicators.

In particular, \cite{BinSub} applies their result to the setting of triangles in $G(n,p)$ where a hypergraph is formed on the vertex set ${[n] \choose 2}$ with a 3-uniform hyperedge for each of the ${n \choose 3}$ possible triangles on $n$ vertices. Applying Theorem 10, they obtain a precise asymptotic expression for the probability of $G(n,p)$ being triangle-free.

In our case, we condition on a triangle configuration $F$, so we define our hypergraph $\cH$ on the set of pairs $uv \notin E(G_F)$. A (hyper)edge in $\cH$ consists of a set of 1, 2, or 3 potential edges in ${V \choose 2} - E(G_F)$ which form a new triangle when added to $G_F$. The hypergraph $\cH$ is no longer 3-uniform and in particular, we index the hyperedges by the potential triangles that they will form, which makes $\cH$ a multi-hypergraph. For configurations in $\cF_{k,\typ}$, however, the multiplicity of each edge is at most two: every set of two or three edges uniquely determines what triangle it completes, while every singleton edge is contained in at most two triangles by typicality.

We now introduce the codegree conditions required for \cite[Theorem 10]{BinSub}. Fix a configuration $F\in \cF_{k,\typ}$. Let $\Omega \subset V(\cH)$ be a set of order 1 or 2, and
define the {\em $j$th codegree of $\Omega$} to be 
\begin{align}
d_j^F(\Omega)
:=
\left|
\left\{
T:\;
\text{$T$ is a triangle in $K_n$, }\Omega\subseteq E(T) \setminus E(G_F),\ \text{and } |E(T)\setminus E(G_F)|=|\Omega| + j
\right\}
\right|.
\end{align}
Thus $d_j^F(\Omega)$ counts the number of hyperedges in $\cH$ consisting of $\Omega$ along with $j$ other vertices of $\cH$, or equivalently the number of potential triangles in $G$ containing $\Omega$ that are not already present in $F$ and that would be completed by adding exactly $j$ further edges. Here $j \in \{1, 2\}$.
We then set
\begin{align}
\cD_F(p):=\max_{j\in\{1,2\}}\ \max_{\emptyset\neq \Omega\subseteq {V \choose 2} \setminus E(G_F)} d_j^F(\Omega)p^j.
\end{align}
To define the joint cumulants and joint moments, given $F$, consider the dependency graph where two potential triangles are adjacent if their sets of missing edges intersect. A {\em triangle cluster} is a collection of potential triangles that is connected in the dependency graph.

    The {\em joint moment} of a set $A = \{Z_1, \dots, Z_n\}$ of random variables is
    $$\Delta(A):=\E[Z_1\dots Z_n]\ ,$$
   and the {\em joint cumulant} of $A$ is 
   $$\kappa(A) := \sum_{\pi\in \Pi (A)} (|\pi|-1)!(-1)^{|\pi|-1}\prod_{P\in \pi} \Delta(P)$$
   where $\Pi(A)$ is the set of all non-empty partitions of $A$. For example, $\kappa(Z) = \E[Z]$ and $\kappa(Z_1, Z_2) = \mathrm{Cov}(Z_1, Z_2)$.

Now let $\cC_i$ be the set of all triangle clusters of order $i$, and let $X_i$ be the indicator random variable for the $i$th hyperedge of $\cH$, $1 \leq i \leq {n \choose 3} - X(G_F)$. We define the following quantities.
    \begin{align}
        \kappa_i &:= \sum_{A  \in \cC_i} \kappa(\{X_j : j \in A\})\\
        \Delta_i &:= \sum_{A \in \cC_i} \Delta(\{X_j : j \in A\})\\
        \delta &:= \sum_i \E[X_i]^2\ .
    \end{align}

\newpage

\begin{theorem}[\cite{BinSub}, Theorem 10]
\label{theorem: BinSub}
If $p =o(1)$ and $\cD_F(p)=O(1)$ as $n \to \infty$, then for every fixed integer $c\in\mathds{N}$,
\begin{align}
\cP(F)
=
\exp\left(
-\kappa_1+\kappa_2-\kappa_3+\cdots+(-1)^c\kappa_c+O\!\left(\delta+\Delta_{c+1}\right)
\right).
\end{align}
\end{theorem}
Although \cite[Theorem 10]{BinSub} is stated for simple hypergraphs, the result extends to multi-hypergraphs with bounded hyperedge multiplicity. Indeed, repeated edges may be treated as having distinct indices throughout, and in this setting the proof of the main probabilistic estimates in \cite{BinSub} remains unchanged while the combinatorial estimates used to prove Theorem 10 require minor modifications in the form of additional constant factors depending on the maximum hyperedge multiplicity. Since the hyperedges of $\cH$ have multiplicity at most 2, we may therefore apply the theorem to $\cH$ in order to prove \zcref[S]{lemma: formula}.

\begin{proof}[Proof of {\zcref[S]{lemma: formula}}]
We first verify the codegree hypothesis of \zcref[S]{theorem: BinSub}. If $j=2$ and $d_2^F(\Omega)\neq 0$, then $|\Omega|=1$, and 
a fixed edge is contained in at most $n-2$ triangles. Hence,
\begin{align}
d_2^F(\Omega)p^2\le np^2=o(1).
\end{align}
If $j=1$ and $|\Omega|=2$, then $\Omega$ determines at most one triangle, so
\begin{align}
d_1^F(\Omega)p\le p=o(1).
\end{align}
Finally, if $j=1$ and $\Omega=\{uv\}$, every triangle counted by $d_1^F(\Omega)$ contains an edge of $G_F$ incident to either $u$ or $v$. Since $F\in\cF_{k,\typ}$, \zcref[S]{lemma:bdd-deg} says that every vertex of $F$ has degree at most $\log n$. Therefore
\begin{align}
d_1^F(\Omega)\le \deg_F(u)+\deg_F(v)\le 2\log n,
\end{align}
and so
\begin{align}
d_1^F(\Omega)p\le 2p\log n=o(1).
\end{align}
Thus, $\cD_F(p)=o(1)$ so the hypothesis of \zcref[S]{theorem: BinSub} is satisfied. We now compute the $\kappa_i$ terms for $i = 1, 2, 3$. Let
$$D(F) := \sum_{v \in V} \deg_F(v)^2$$
and note that \zcref[S]{lemma:bdd-deg} implies 
\begin{equation}\label{eq:square-deg-bd}
    D(F) \leq \max_{v \in V} \deg_F(v) \sum_{v \in V} \deg_F(v) \leq 2e_F \log n \leq 6k\log n\ .
\end{equation}

\noindent{\bf Computing $\kappa_1$:}   
    Recall that $\kappa_1$ is the expected number of triangles appearing in $G$ but not in $F$. Let $N_i$ be the number of triangles with exactly $i$ edges in $G_F$, where $i \in \{1, 2, 3\}$. 
Since $\deg_F(v) > 0$ implies $\deg_F(v) \geq 2$ for all $v \in V(G)$, each such vertex is contained in ${\deg_F(v) \choose 2}$ distinct triples with at least two edges covered by $G_F$, whereas the triples containing three edges of $G_F$ are exactly the $k$ triangles of $F$ by assumption. Therefore,
\begin{align}
    N_2 &= \sum_{v \in V} {\deg_F(v) \choose 2} - 3k = \frac12 D(F) - 3k - e_F, \\
    N_1 &= e_F(n-2) -2N_2 - 3k = e_Fn+3k-D(F).
\end{align}
The probability of a triangle with $i$ covered edges is $p^{3-i}$, so we have
\begin{align}
    \kappa_1 &= \left({n \choose 3}-k\right)p^3 + N_2(p-p^3) + N_1(p^2-p^3) \\
    &= \frac{n^3p^3}{6} -3kp - e_Fp + e_Fnp^2 + \frac{p}{2}D(F)+ o(1)\ .
\end{align}
Indeed, by \eqref{eq:square-deg-bd} and the assumptions $k = O(n^3p^3)$ and $p = o(n^{-2/3})$, the discarded terms are bounded by $O(n^2p^3 + kp^2\log n + knp^3) = o(1)$.

\medskip

\noindent{\bf Computing $\kappa_2$:}
Say $T_i$ and $T_j$ are potential triangles. Then $\kappa(\{X_i, X_j\}) = 0$ unless they share an edge outside $G_F$. Thus it suffices to consider triangle pairs of the form $T_i = \{u,v,x\}$ and $T_j = \{u,v,y\}$ where $x \neq y$ and $uv \notin G_F$. Recall that $\kappa_2$ is a sum over unordered pairs.

Let $a_i = |E(T_i) \cap E(G_F)|$ and $a_j = |E(T_j) \cap E(G_F)|$.
Then 
$$\mathrm{Cov}(X_i, X_j) = p^{5-(a_i+a_j)} - p^{6-(a_i+a_j)} = (1-p)p^{5-(a_i+a_j)}\ .$$ 
The possible cases and their total contributions to $\kappa_2$ are as follows.
$$
\begin{array}{c|c|c}
(a_i,a_j)
&\text{additional condition}
&\text{contribution to }\kappa_2\\ \hline
(2,2)&& p(Q+2R+3S)+o(1)\\
(2,1)&&O(kp^2\log^2n)=o(1)\\
(1,1)&\text{the two covered edges are disjoint}&O(k^2p^3)=o(1)\\
(1,1)&\text{the two covered edges are adjacent}&\frac12np^3D(F)+o(1)\\
(2,0)&&\frac12np^3D(F)+o(1)\\
(1,0)&&2e_Fn^2p^4+o(1)\\
(0,0)&&\frac14n^4p^5+o(1).
\end{array}
$$

In the case $(a_i, a_j) = (2,2)$, typicality implies that the four covered edges extend to a diamond in $F$, and each diamond, piglet, and 3-book gives 1, 2, and 3 such pairs of triangles, respectively.
This gives a contribution of $(p-p^2)(Q + 2R + 3S) = p(Q + 2R + 3S) + o(1)$.

For $(a_i, a_j) = (2,1)$, the number of triangle pairs is at most $\sum_{uv \notin E(G_F)} (\deg_F(u) + \deg_F(v))(|\partial_F(u) \cap \partial_F(v)|) = O(k \log^2 n)$ (where $\partial_F(x)$ denotes the neighborhood of $x$ in $G_F$) while for $(a_i, a_j) = (1,1)$ with disjoint edges covered, there are $O(e_F^2) = O(k^2)$ possibilities. These give the two negligible contributions to $\kappa_2$. 

For the remaining cases, the counts are obtained by first choosing the covered edges as shown in Figure~\ref{fig: kappa_2} and then choosing the remaining vertices. Using \eqref{eq:square-deg-bd}, we get
\begin{align*}
    N_{(1,1),\mathrm{adj}} &= n\sum_{v \in V} {\deg_F(v) \choose 2} + o(p^{-3}) \\
    N_{(2,0)} &= n\sum_{v \in V} {\deg_F(v) \choose 2} + o(p^{-3}) \\
    N_{(1,0)} &= 2e_Fn^2 + o(p^{-4}) \\
    N_{(0,0)} &= \frac14 n^4 + o(p^{-5})
\end{align*}
The error terms arise from excluding choices in the $(O(\log n))$-sized neighborhoods of already chosen vertices or choices containing additional edges of $G_F$, using $e_F = O(k)$, $\deg_F(v) \leq \log n$, and $k = O(n^3p^3)$.

\begin{figure}[h]
\centering
\begin{tikzpicture}[scale=0.9]
    \coordinate[label=below:$u$] (u) at (2,0);
    \coordinate[label=above:$v$] (v) at (3,2);
    \coordinate[label=below:$y$] (y) at (4,0);
    \coordinate[label=above:$x$] (x) at (1,2);
    \foreach \P in {u, v, x, y} {\vertex{\P}}
    \draw[line width=.75mm, black] (2,0) -- (1,2);
    \draw[] (1,2) -- (2,0) -- (3,2) -- cycle;
    \draw[] (2,0) -- (3,2) -- (4,0) -- cycle;
    \draw[line width=.75mm, black] (2,0) -- (4,0);
\end{tikzpicture}
\begin{tikzpicture}[scale=0.9]
    \coordinate[label=below:$u$] (u) at (2,0);
    \coordinate[label=above:$v$] (v) at (3,2);
    \coordinate[label=below:$y$] (y) at (4,0);
    \coordinate[label=above:$x$] (x) at (1,2);
    \foreach \P in {u, v, x, y} {\vertex{\P}}
    \draw[line width=.75mm, black] (2,0) -- (1,2);
    \draw[] (1,2) -- (2,0) -- (3,2) -- cycle;
    \draw[] (2,0) -- (3,2) -- (4,0) -- cycle;
    \draw[line width=.75mm, black] (1,2) -- (3,2);
\end{tikzpicture}
\begin{tikzpicture}[scale=0.9]
    \coordinate[label=below:$u$] (u) at (2,0);
    \coordinate[label=above:$v$] (v) at (3,2);
    \coordinate[label=below:$y$] (y) at (4,0);
    \coordinate[label=above:$x$] (x) at (1,2);
    \foreach \P in {u, v, x, y} {\vertex{\P}}
    \draw[line width=.75mm, black] (2,0) -- (1,2);
    \draw[] (1,2) -- (2,0) -- (3,2) -- cycle;
    \draw[] (2,0) -- (3,2) -- (4,0) -- cycle;
\end{tikzpicture}
\begin{tikzpicture}[scale=0.9]
    \coordinate[label=below:$u$] (u) at (2,0);
    \coordinate[label=above:$v$] (v) at (3,2);
    \coordinate[label=below:$y$] (y) at (4,0);
    \coordinate[label=above:$x$] (x) at (1,2);
    \foreach \P in {u, v, x, y} {\vertex{\P}}
    \draw[] (1,2) -- (2,0) -- (3,2) -- cycle;
    \draw[] (2,0) -- (3,2) -- (4,0) -- cycle;
\end{tikzpicture}
  \caption{The cluster types with nonnegligible cumulant contribution; bold edges are those contained in $G_F$.}
  \label{fig: kappa_2}
\end{figure}
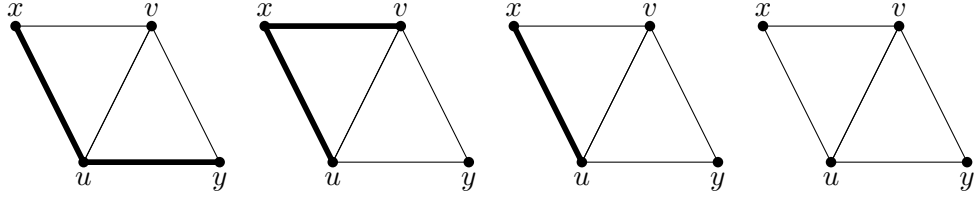

By noting that $\sum_{v \in V} {\deg_F(v) \choose 2} = \frac12 D(F) - e_F$, we have
\begin{align}
    \kappa_2 = nD(F)p^3+2e_Fn^2p^4+\frac{n^4p^5}{4}+pQ+2pR+3pS+o(1).
\end{align}

\noindent{\bf Computing $\kappa_3$:} We now consider connected clusters of 3 potential triangles $T_h, T_i$, and $T_j$.
Notice that $\kappa(X_h, X_i, X_j) = 0$ unless $T_h, T_i, T_j$ form a piglet, a 3-book, or a $K_4$. We first show that the total contribution to $\kappa_3$ from clusters containing at least one edge of $G_F$ is $o(1)$. It will be convenient to bound the joint moment instead of the joint cumulant, as in our case $|\kappa(X_h, X_i, X_j)| \leq 6\E[X_hX_iX_j]$.

Suppose first that the triangles form a piglet or a 3-book $H$. Let $J := E(G_F) \cap E(H)$. Observe that the two shared edges of a piglet or the one shared edge of a 3-book must lie outside $G_F$; consequently, the possible covered edges in a piglet form a subgraph of $C_5$ and in a 3-book form subgraph of $K_{2,3}$. The case $J = K_{2,3}$ is excluded by the assumption that $F$ is typical. 

If a connected graph on $s$ vertices is required to be a subgraph of $G_F$, then there are $O(k (\log n)^{s-2})$ possible embeddings: choose one of its edges in $O(k)$ ways and expose the remaining vertices along a spanning tree, recalling that the degrees are bounded by $\log n$. For disconnected $J$, we apply the same embedding bound to each component. The resulting contributions to $\kappa_3$ are given by the following.
$$
\begin{array}{c|c}
e(J) & \text{total joint-moment contribution}\\ \hline
1 & O(kn^3p^6)\\
2 & O(k(\log n)n^2p^5+k^2np^5)\\
3 & O(k(\log^2n)np^4+k^2(\log n)p^4)\\
4 & O(k(\log^3n)p^3+knp^3)\\
5 & O(k(\log^3n)p^2).
\end{array}
$$

By Assumption~\ref{assumption1}, $k = O(n^3p^3)$ and $p = o(n^{-2/3})$, so every entry in the table is $o(1)$. 

For the $K_4$ case, at most four of the six edges can be covered by $G_F$. As before, we consider the different cases of $J = E(K_4) \cap E(G_F)$ based on $e(J)$ and the number of connected components; applying the previous embedding bound gives a total contribution of
$$O\left(kn^2p^5+k(\log n)np^4+k^2p^4+k(\log^2n)p^3+knp^3+k(\log^2n)p^2\right)
    =o(1)\ .$$
Hence the total joint cumulant of all 3-clusters that intersect $G_F$ is $o(1)$.

It remains to consider the clusters which are edge-disjoint from $G_F$. 
The total number of piglets on $n$ vertices is $60{n\choose 5}=(1+o(1)) \frac{n^5}{2}$, and the total number of 3-books on $n$ vertices is $10{n\choose 5}=(1+o(1)) \frac{n^5}{12}$. The number of either type with at least one edge covered by $G_F$ is $O(kn^3) = o(n^5)$.
The joint moment for all three triangles is $p^7$ whereas every other term in the cumulant is $O(p^8)$. The remaining $K_4$ clusters contribute $O(n^4p^6) = o(1)$.  

Putting everything together gives us
\begin{align}
    \kappa_3=(1+o(1))\left(\frac12 + \frac{1}{12}\right)n^5p^7 = \frac{7}{12}n^5p^7+o(1).
\end{align}

\noindent{\bf Bounding the error:} We first bound $\delta$. From the computations for $\kappa_1$, we know the number of potential triangles missing 1, 2, or 3 edges is $O(k \log n), O(kn)$, and $O(n^3)$, respectively. Thus,
$$\delta = O(k \log n p^2 + knp^4 + n^3p^6) = o(1)\ .$$
Next, we show $\Delta_4 = o(1)$. Observe that a connected 4-cluster meeting $G_F$ must contain a connected 3-cluster, say $C$, that also meets $G_F$. For a fixed $C$, consider the number of ways to extend $C$ by a fourth hyperedge $T$ of $\cH$. There are $O(1)$ choices for $T$ which are contained in the union of the hyperedges of $C$, as $\cH$ has bounded hyperedge multiplicity. Else, $T$ extends $C$ by $j \in \{1, 2\}$ new vertices of $\cH$, and the weighted contribution of such $T$ is at most $d_j^F(\Omega)p^j \leq \cD_F(p) = o(1)$ where $\Omega = T \cap (\bigcup C)$. As there are again a constant number of choices for the vertices of $\Omega$, this along with the previous case bounding the total joint moment of 3-clusters intersecting $G_F$ shows that the total contribution of 4-clusters meeting $G_F$ is $o(1)$. 

Now consider a 4-cluster that is edge-disjoint from $G_F$. As previously computed, a connected 3-subcluster has total joint moment $O(n^5p^7 + n^4p^6)$ where the first term comes from piglets and 3-books, and the second term comes from $K_4$'s. If the fourth triangle is not contained in the union of the 3-subcluster, then it adds either one or two new edges, so the weighted number of extensions of the 3-subcluster is $O(p+np^2)$. Else, it must be the case that 4-cluster is a $K_4$, which contributes $O(n^4p^6)$ to the joint moment. Therefore,
$$\Delta_4 = O(n^5p^7 + n^4p^6)O(p+np^2) + O(n^4p^6) + o(1) = o(1)\ .$$

Applying \zcref[S]{theorem: BinSub} for $c = 3$ gives the claim.
\end{proof}

\subsection{Conditioning on fixed component counts}\label{subsec:fixed-component-counts}

For integers $k$, $Q$, $R$, $S$ with $2Q+3R+3S\le k$, let $\cF_{k,Q,R,S} \subseteq \cF_{k, \elem}$ be the set of triangle configurations on $V$ consisting of exactly $Q$ diamonds, $R$ piglets, $S$ 3-books and $t = k-2Q-3R-3S$ isolated triangles, and let $\cF_{k,Q,R,S,\typ} = \cF_{k,Q,R,S} \cap \cF_{k,\typ}$. Let $\nufix$ denote the uniform distribution on $\cF_{k,Q,R,S,\typ}$. 

\medskip

For the remainder of this subsection, fix $k$, fix $Q, R, S$ such that $2Q + 3R + 3S \leq k$, and let $\nu = \nufix$. We will analyze $\nu$ by comparing it to the following independent embedding model.

\medskip

Let $m= t+Q+R+S= k-Q-2R-2S$ denote the total number of components in the triangle-dual graph. Let $H_1,\ldots,H_m$ be an ordered list of the $m$ components with types specified by $Q,R,S,t$. Give every $H_i$ a fixed labeling of its vertices, and independently for each $i$ choose a uniformly random injection $\phi_i:V(H_i)\hookrightarrow[n]$.
Let
\begin{equation}
    \widetilde\nu=\widetilde\nu_{k,Q,R,S}^{\ind}
\end{equation}
denote the resulting distribution on the tuple $\boldsymbol\phi=(\phi_1,\ldots,\phi_m)$. Let $\tE = \E_{\tnu}$ denote expectation in this model.

Let $I_i=\phi_i(V(H_i))$ and let $\mathcal T(\boldsymbol\phi)$ be the set of images of all the
constituent triangles of $H_1,\ldots,H_m$. Define
\begin{align}
\cE := \left\{\mathcal T(\boldsymbol\phi)\in\cF_{k,Q,R,S,\typ}\right\}.
\label{eq:def-E}
\end{align}
Thus $\cE$ is the event that the independently embedded components form a valid triangle configuration with the prescribed component counts (meaning, the embeddings of different components do not share edges, and the embeddings do not create additional triangles) and that this configuration is typical. We first show that this event holds whp.

\begin{lemma}\label{lemma: indept-conditional}
Uniformly over $(Q, R, S)$ such that $2Q + 3R + 3S \leq k$, 
    $$\tnu(\cE) = 1-o(1)\ .$$
Moreover, the distribution of $\mathcal T(\boldsymbol\phi)$ conditioned on $\cE$ is distributed exactly according to $\nu$. 
\end{lemma}

\begin{proof}
Recall that $r = \lceil\frac{\log n}{6}\rceil$ and set $L = \lceil\frac{r}{3}\rceil$.
Consider the following ``bad'' events:
\begin{align*}
    \cB_2 &:= \{\text{there exist } i \neq j \text{ such that } |I_i \cap I_j| \geq 2\}, \\
    \cB_q &:= \{\text{there exist distinct } i_1, \cdots, i_q \text{ and distinct } v_1, \cdots, v_q \text{ such that } v_j \in I_{i_j} \cap I_{i_{j+1}} \forall j \in [q]\}\\ &\qquad\text{ for } q \in \{3, 4\} \text{ and where } i_{q+1} = i_1,\\
    \cB_5 &:= \{\max_{v \in V} N_v \geq L\}\ .
\end{align*}
where $N_v = |\{i \in [m] : v \in I_i\}|$. 
Let $\mathcal G = \left(\bigcup_{i=2}^5 \cB_i\right)^c$. We claim that $\mathcal G \subseteq \cE$, in which case it would suffice to show that $\tnu(\mathcal G) = 1-o(1)$.

Indeed, on $\mathcal G$, every pair of distinct embedded components intersect in at most one vertex and therefore share no edges. Now suppose that the union of $\mathcal T(\boldsymbol\phi)$ induces a triangle $T$ that does not come from $H_i$ for any $i$. Then the edges of $T$ must come from at least two distinct components. If two of the edges are from $H_i$ and one is from $H_j$ for $i \neq j$, then $|I_i \cap I_j| \geq 2$, contradicting $\cB_2^c$. If the three edges come from three distinct components, then these and the vertices of $T$ contradict $\cB_3^c$. Thus, for $\mathcal T(\boldsymbol\phi) \in \mathcal G$, the union of $\mathcal T(\boldsymbol\phi)$ has exactly $k$ triangles.

It remains to check typicality. Since distinct components share no edge and every $H_i$ contains at most three triangles, the union contains neither a $K_4$ nor four edge-connected triangles. If the union contained $r$ edge-disjoint triangles having exactly one common vertex, then at least $L$ different components would contain that vertex, contradicting
$\cB_5^c$.

Finally, suppose that four triangles in the union formed one of the forbidden configurations whose triangle-intersection graph is $C_4$. A direct check of the four allowed component types shows that every two triangles belonging to the same component intersect in at least one vertex. Hence a component cannot contain two opposite triangles of this $C_4$. Partition the four triangles into maximal consecutive blocks belonging to the same edge-component. If there are two blocks, the two corresponding components intersect at the two distinct boundary vertices of the cycle, contradicting $\cB_2^c$. If there are three or four blocks, the corresponding components together with the distinct boundary vertices contradict
$\cB_3^c$ or $\cB_4^c$. This proves
$\mathcal G\subseteq\cE$. 

Now we bound the probability of each bad event. First, as all the $H_i$ have bounded order, for each fixed pair $i \neq j$, we have
$$\tnu(|I_i \cap I_j| \geq 2) = O(n^{-2})$$
and consequently,
\begin{equation}\label{eq: placed together}
\tnu(\cB_2) = O\left(\frac{m^2}{n^2}\right) = o(1)\ .
\end{equation}
For $\cB_3$, given a fixed $i_1, i_2, i_3$ and three pairs of vertices from the component graphs whose images are to be identified, the probability that these components and vertices satisfy $\cB_3$ is $O(n^{-3})$. Since each component has bounded order,
$$\tnu(\cB_3) = O\left(\frac{m^3}{n^3}\right) = o(1)$$
and by analogous reasoning,
$$\tnu(\cB_4) = O\left(\frac{m^4}{n^4}\right) = o(1)\ .$$
For $\cB_5$, notice that the indicators $\mathds{1}_{v \in I_i}$ are independent over $i$, and 
$$\tnu(v \in I_i) = \frac{|V(H_i)|}{n} \leq \frac{5}{n}\ .$$
Thus $N_v$ is stochastically dominated by $\Bin(m,5/n)$, and hence
$$\tnu(\cB_5) \leq n\left(\frac{5em}{nL}\right)^L =o(1)$$
where we use $m/n=o(1)$ and $L=\Theta(\log n)$.

Altogether, this shows that $\tnu(\mathcal G) = 1-o(1)$ as desired.

For the second part of the claim, we consider the distribution $\tnu(\cdot \mid \cE)$. Every $F\in\cF_{k,Q,R,S,\typ}$ has the same number of preimages under the independent embedding construction. Indeed, the edge-components of $F$ may be assigned to the corresponding positions in the ordered list in
$t!\,Q!\,R!\,S!$ ways, and each resulting component image has
$|\operatorname{Aut}(H_i)|$ embeddings from its corresponding abstract graph. Thus the number of preimages of $F$ is
$$t!\,Q!\,R!\,S!
\prod_{i=1}^m|\operatorname{Aut}(H_i)|\ ,$$
which is independent of $F$. It follows that the distribution of $\mathcal T(\boldsymbol\phi)$ conditional on $\cE$ is exactly $\nu$.
\end{proof}

Let 
$$\td(v) := \sum_{i=1}^m \sum_{x \in V(H_i)} \deg_{H_i}(x)\mathds{1}_{\phi_i(x) = v}$$
be the total degree of $v$ once all components have been embedded, and let
$$\tD := \sum_{v \in V} \td(v)^2\ .$$

\begin{lemma}
\label{lemma: indept-degree-moments}
In the independent embedding model,
\begin{align}
    \tE[\tD] = 12k + 2Q + 6R + 8S +\frac{4}{n}\left((3k-Q-2R-2S)^2-(9k+7Q+22R+22S)\right),
\end{align}
and 
$$\Var_{\tnu}(\tD) = O\left(\frac{m^2}{n}\right)\ .$$

\end{lemma}

\begin{proof}
Expanding $\tD$ gives us
    \begin{equation}
\tD = \sum_{i=1}^m \sum_{x \in V(H_i)}\deg_{H_i}(x)^2 + \sum_{i \neq j} \sum_{x \in V(H_i)} \sum_{y \in V(H_j)} \deg_{H_i}(x)\deg_{H_j}(y)\mathds 1_{\phi_i(x) = \phi_j(y)}
\end{equation}
The first term is deterministic, so we can compute the expectation as $12(k-2Q-3R-3S) + 26Q + 42R + 44S = 12k + 2Q + 6R + 8S$. For $i \neq j$, 
$$\tnu(\phi_i(x) = \phi_j(y)) = \frac1n$$
so the expectation of the second term becomes \begin{align*} \frac1n \sum_{i \neq j} \sum_{x \in V(H_i)}\deg_{H_i}(x) \sum_{y \in V(H_j)} \deg_{H_j}(y) &= \frac{1}{n}\sum_{i \neq j} \left(2e(H_i) \cdot 2e(H_j)\right)\\
&= \frac4n\left(\left(\sum_{i} e(H_i)\right)^2 - \sum_{i} e(H_i)^2\right)
\end{align*}
which we can again compute deterministically: $\sum_{i} e(H_i) = e_F = 3k - Q - 2R - 2S$ while $\sum_{i} e(H_i)^2 = 9(k-2Q-3R-3S) + 25Q + 49R + 49S = 9k + 7Q + 22R + 22S$. Thus, we have
$$\tE\left[\sum_{i \neq j} \sum_{x \in V(H_i)} \sum_{y \in V(H_j)} \deg_{H_i}(x)\deg_{H_j}(y)\mathds 1_{\phi_i(x) = \phi_j(y)}\right] = \frac{4}{n}(e_F^2 - (9k+7Q+22R+22S))\ .$$
Putting everything together gives precisely the first part of the claim.

For the variance computation, let
$$X_{ij} := \sum_{x \in V(H_i)} \sum_{y \in V(H_j)} \deg_{H_i}(x)\deg_{H_j}(y)\mathds 1_{\phi_i(x) = \phi_j(y)}\ .$$
Since $\tD$ is the sum of a deterministic quantity with $2\sum_{i < j} X_{ij}$, we have
$$\Var_{\tnu}(\tD) = 4\Var_{\tnu}\left(\sum_{i<j}X_{ij}\right)\ .$$
Since the component graphs have bounded order and degree, $X_{ij} = O(1)$ deterministically, and $X_{ij} \neq 0$ if and only if $I_i \cap I_j \neq \emptyset$. Since
$\tnu(I_i \cap I_j \neq \emptyset) = O\left(\frac{1}{n}\right)$
it follows that
$$\tE[X_{ij}^2] = O\left(\frac1n\right)$$
and thus
$$\Var_{\tnu}(X_{ij}) = O\left(\frac1n\right)\ .$$
We next show that the variables $X_{ij}$ are pairwise uncorrelated. If $\{i,j\} \cap \{a,b\} = \emptyset$, then $X_{ij}$ and $X_{ab}$ are independent, so it remains to consider pairs of the form $X_{ij}$ and $X_{i\ell}$ for $j \neq \ell$. Conditioned on $\phi_i$, the random variables $X_{ij}$ and $X_{i\ell}$ are independent, since $\phi_j$ and $\phi_{\ell}$ are independent. But we also have
$$\tE[X_{ij} \mid \phi_i] = \frac{4}{n} e(H_i)e(H_j)$$
which is independent of $\phi_i$, and similarly
$$\tE[X_{i\ell} \mid \phi_i] = \frac{4}{n}e(H_i)e(H_{\ell})\ .$$
Thus, 
$$\Cov_{\tnu}(X_{ij},X_{i\ell})=\tE[\Cov_{\tnu}(X_{ij},X_{i\ell}\mid \phi_i)]+\Cov_{\tnu}(\tE[X_{ij}\mid \phi_i],\tE[X_{i\ell}\mid \phi_i])=0+0\ .$$
Combining these estimates, 
$$\Var_{\tnu}(\tD) = 4 \sum_{i<j} \Var_{\tnu}(X_{ij}) = O\left(\frac{m^2}{n}\right)$$
as claimed.
\end{proof}

We transfer these results back to the measure of interest, $\nu = \nu_{k,Q,R,S,\typ}^{\unif}$. As in the previous section, let
$$D(F) = \sum_{v \in V} \deg_F(v)^2\ .$$

\begin{lemma}
    \label{lemma: fixed-expectation-variance}
Uniformly over $Q, R, S$ satisfying $2Q+3R+3S\leq k$, we have
    $$\E_{\nu}[D(F)] = 12k + 2Q + 6R + 8S +\frac{4}{n}\left((3k-Q-2R-2S)^2-(9k+7Q+22R+22S)\right) + o\left(\frac{m}{\sqrt{n}}\right)\ ,$$
    and
    $$\Var_{\nu}(D(F)) = O\left(\frac{m^2}{n}\right)\ .$$
\end{lemma}

\begin{proof}
 From \zcref[S]{lemma: indept-conditional}, the configuration $\mathcal T(\boldsymbol{\phi})$ conditioned on $\cE$ is distributed according to $\nu$. On $\cE$, we also have $\tD = D(F)$. Thus,
     $$\E_{\nu}[D(F)] = \tE[\tD \mid \cE]\ .$$
To compare the expected value under the two different distributions, we apply the Cauchy-Schwarz inequality. 
\begin{align*}
    |\E_{\nu}[D(F)] - \tE[\tD] | &= |\tE[\tD \mid \cE] - \tE[\tD]|\\
    &= \frac{|\tE[(\tD - \tE[\tD])\mathds 1_{\cE^c}]}{\tnu(\cE)}\\
    &\leq \frac{\sqrt{\Var_{\tnu}(\tD)}\sqrt{\tnu(\cE^c)}}{\tnu(\cE)}\\
    &= o\left(\frac{m}{\sqrt{n}}\right)
\end{align*}
where in the last line, we use \zcref[S]{lemma: indept-degree-moments} and \zcref[S]{lemma: indept-conditional}. Combining this with the computation of $\tE[\tD]$ from \zcref[S]{lemma: indept-degree-moments} gives the first part of the claim. 

Similarly, we have for the variance
\begin{align*} \Var_{\nu}(D(F)) &= \Var_{\tnu}(\tD \mid \cE)\\
&\leq \tE[(\tD - \tE[\tD])^2 \mid \cE]\\
&\leq \frac{\Var_{\tnu}(\tD)}{\tnu(\cE)}\\
&= O\left(\frac{m^2}{n}\right)\ .
\end{align*}

\end{proof}

We use these results to analyze the term involving the sum of squared degrees in \eqref{eq:formula_M}.

\begin{lemma}
\label{lemma:expectation-exchange}
For $0\leq\lambda=O(p)$,
uniformly over $Q,R,S$ satisfying $2Q+3R+3S\leq k$,
\[
\E_\nu[e^{-\lambda D(F)}]
=
(1+o(1))
\exp\left(-\lambda\E_\nu[D(F)]\right).
\]
In particular, the assertion holds for
$\lambda=p/2-np^3$.
\end{lemma}

\begin{proof}
We start by proving the corresponding statement in the independent embedding model $\tnu$. Suppose we expose the embeddings $\phi_1, \dots, \phi_m$ sequentially and let $\cH_j$ be the $\sigma$-algebra generated by $\phi_1, \dots, \phi_j$, that is, the information revealed by the first $j$ component embeddings. Let
$$d_j(v) := \sum_{i \leq j} \sum_{x \in V(H_i)} \deg_{H_i}(x) \mathds 1_{\phi_i(x) = v}$$
and let
$$Z_j := \sum_{v \in V} d_j(v)^2 - \sum_{i \leq j} \sum_{x \in V(H_i)} \deg_{H_i}(x)^2\ .$$
Thus $Z_j$ counts the cross-component contribution to the sum of squared degrees. In particular, 
$$Z_m = \tD - \sum_{i=1}^m \sum_{x \in V(H_i)} \deg_{H_i}(x)^2\ .$$
The subtracted quantity is deterministic, so
\begin{equation}\label{eqn:Z-D}
Z_m - \tE[Z_m] = \tD - \tE[\tD]\ .
\end{equation}
Let $Y_i = Z_i - Z_{i-1}$. Then
$$Y_i = 2\sum_{x \in V(H_i)}\deg_{H_i}(x)d_{i-1}(\phi_i(x))\ .$$
Since each $\phi_i(x)$ has a uniform marginal distribution on $[n]$,
\begin{align*}
    \mu_i &:= \tE[Y_i \mid \cH_{i-1}]\\
    &= \frac{2}{n} \left(\sum_{x \in V(H_i)} \deg_{H_i}(x)\right)\left(\sum_{v \in V} d_{i-1}(v)\right)\\
    &= O\left(\frac{i}{n}\right)
\end{align*}
where in the last step we use the observation $\sum_{v \in V} d_{i-1}(v) = 2 \sum_{j < i} e(H_j) = O(i)$.
Moreover, $\mu_i$ is deterministic. 

By Cauchy--Schwarz and the bounded size and degrees of the components,
there is an absolute constant $C_0$ such that
$$Y_i^2 \leq 4\left(\sum_{x \in V(H_i)}\deg_{H_i}(x)^2\right)\left(\sum_{x \in V(H_i)}d_{i-1}(\phi_i(x))^2\right) \leq C_0\sum_{x\in V(H_i)}d_{i-1}(\phi_i(x))^2\ .$$
Taking conditional expectations gives, for some constants $C_1, C_2 > 0$,
\begin{align*}
\widetilde\E[Y_i^2\mid\mathcal H_{i-1}]
&\leq C_0\sum_{x \in V(H_i)}\tE[d_{i-1}(\phi_i(x))^2\mid \cH_{i-1}]\\
&\leq \frac{C_1}{n}
\sum_{v\in V}d_{i-1}(v)^2\\
&= \frac{C_1}{n}\left(Z_{i-1}+C_2i\right).
\end{align*} 
where in the last line, we rearrange the definition of $Z_{i-1}$ to substitute for $\sum_{v \in V} d_{i-1}(v)^2$. 

We can now bound the moment-generating function under $\widetilde\nu$. For $u\geq0$, we know $e^{-u}\leq1-u+\frac{u^2}{2}$, and so
for every $\theta\geq0$,
\begin{align*}
    \tE\left[e^{-\theta Y_i}\mid\mathcal H_{i-1}\right]
&\leq
1-\theta\mu_i
+\frac{\theta^2}{2}
\tE[Y_i^2\mid\mathcal H_{i-1}]\\
&\leq \exp\left(-\theta\mu_i + \frac{C_3\theta^2}{n}(Z_{i-1}+i)\right)\ .
\end{align*}
Now recall that $Z_i = Z_{i-1} + Y_i$. As $Z_{i-1}$ is $\cH_{i-1}$-measurable, by the Law of Total Expectation,
$$\tE[e^{-\theta Z_i}] = \tE[e^{-\theta Z_{i-1}} \tE[e^{-\theta Y_i} \mid \cH_{i-1}]]\ .$$
It follows that
\begin{align} \tE[e^{-\theta Z_i}] &\leq \tE\left[e^{-\theta Z_{i-1}}\exp\left(-\theta\mu_i + \frac{C_3\theta^2}{n}(Z_{i-1}+i)\right)\right]\\
&\leq \exp\left(-\theta\mu_i+\frac{C_3i\theta^2}{n}\right)\tE\left[\exp\left(-\left(\theta-\frac{C_3\theta^2}{n}\right)Z_{i-1}\right)\right].
\label{eq:Laplace-recursion}
\end{align}
Set $\theta_m=\lambda$ and, recursively for $i=m,\ldots,1$, set
$$\theta_{i-1}:=\theta_i-\frac{C_3\theta_i^2}{n}.$$
Since $\theta_i \leq \lam$, it follows that
$$0\leq\lambda-\theta_i\leq\frac{C_3(m-i)\lambda^2}{n}\ .$$
Moreover, since $\frac{m\lambda}{n}=O(n^2p^4)=o(1)$,
all the $\theta_i$ are nonnegative for sufficiently large $n$. Iterating \eqref{eq:Laplace-recursion} gives
$$\log\tE[e^{-\lambda Z_m}]\leq-\sum_{i=1}^m\theta_i\mu_i+\frac{C_3}{n}\sum_{i=1}^mi\theta_i^2\ .$$
To bound the first term on the right-hand side, we write $-\sum_{i=1}^m\theta_i\mu_i = -\lam\sum_{i=1}^m \mu_i + \sum_{i=1}^m (\lam - \theta_i)\mu_i$. By the bounds previously shown on $\mu_i$ and $\lam-\theta_i$, we have
$$\sum_{i=1}^m (\lam-\theta_i)\mu_i = O\left(\sum_{i=1}^m\frac{\lam^2}{n^2}i(m-i)\right) = O\left(\frac{\lam^2m^3}{n^2}\right)$$
and so, using $\theta_i \leq \lambda$,
$$\log\tE[e^{-\lambda Z_m}]\leq-\lambda\tE [Z_m]+O\left(\frac{\lambda^2m^2}{n}+\frac{\lambda^2m^3}{n^2}\right)\ .$$
On the other hand, Jensen's inequality gives
$$\log\tE[e^{-\lambda Z_m}]\geq-\lambda\tE[Z_m]\ .$$
Therefore,
$$1\leq\tE\left[e^{-\lambda(Z_m-\tE[Z_m])}
\right] \leq \exp\left(O\left(\frac{\lambda^2m^2}{n}+\frac{\lambda^2m^3}{n^2}\right)\right)\ .$$
Since $m\leq k=O(n^3p^3)$ by Assumption~\ref{assumption1} and $\lambda=O(p)$, we know that $\frac{\lambda^2m^2}{n}=O(n^5p^8)=o(1)$
and $\frac{\lambda^2m^3}{n^2}=O(n^7p^{11})
=o(1)$. 
Together with \eqref{eqn:Z-D}, this shows that
\begin{align}
\tE\left[e^{-\lambda(\tD-\tE[\tD])}\right]=1+o(1).
\label{eq:independent-exchange}
\end{align}
We now transfer the estimate from $\widetilde \nu$ to $\nu$. Recall from the previous proof that $|\E_{\nu}[D(F)] - \tE[\tD] | = o\left(\frac{m}{\sqrt{n}}\right)$. Since $\frac{\lam m}{\sqrt{n}} = O(\frac{pk}{\sqrt{n}}) = O(n^{5/2}p^4) = o(1)$, it follows that $\lam|\E_{\nu}[D(F)] - \tE[\tD] | = o(1)$. Using $\tnu(\cE) = 1-o(1)$ and \eqref{eq:independent-exchange},
\begin{align}
\E_\nu\left[e^{-\lambda(D(F)-\E_\nu [D(F)])}\right]&=
\frac{\tE\left[e^{-\lambda(\tD-\E_{\nu} D(F))}
\mathds 1_{\cE}
\right]}{\tnu(\cE)}\\
&\leq\frac{\exp\left(\lambda\left(\E_\nu[D(F)]-\tE[\tD]
\right)\right)}{\tnu(\cE)}
\tE\left[e^{-\lambda(\tD-\tE[\tD])}\right]\\
&=1+o(1).
\end{align}
On the other hand, Jensen's inequality under $\nu$ gives
$$\E_\nu\left[e^{-\lambda(D(F)-\E_\nu[D(F)])}\right]\geq1\ .$$
Thus
$$\E_\nu\left[e^{-\lambda(D(F)-\E_\nu[D(F)])}
\right]=1+o(1)\ ,$$
which proves the lemma.
\end{proof}

All estimates above are uniform in $Q,R,S$, since the only properties used were $m\leq k$ and the fact that the four possible component types have uniformly bounded order and degrees. These results allow us to simplify the expression in \eqref{eq:formula_M}. 
Recalling Assumption~\ref{assumption1}, we obtain
$$\E_{\nufix}\left[\exp\left(\left(-\frac{p}{2} + np^3\right)D(F)\right)\right] = \exp\left(-6kp - \frac{18k^2p}{n} + O(p(Q+R+S)) + o(1)\right)$$
and thus
  \begin{align}
        \E_{\nufix}\left[\cP(F)\right]=\exp\bigg(&-\frac{n^3p^3}{6} -3kp+e_F (p-np^2+2n^2p^4)+ \frac{n^4p^5}{4} -\frac{7n^5p^7}{12}\\
        & -\frac{18k^2p}{n}+O\left(p(Q+R+S)\right)+o(1) \bigg).\label{eq:fixed-Q-P(F)}
\end{align}

\section{Poisson Approximation for Shared-Edge Components}
\label{section: overlapping edges}
In this section, we show that the probability the triangles of $G(n,p)$ form a configuration in $\cF_{k,\typ}$ is asymptotic to the expression from \zcref[S]{main theorem}. To that end, let $W = W(n, p, k)$ denote the asymptotic expression in \eqref{eq:mainThrm}, with $\theta = \frac{6k}{n^3p^3}$.
$$W = W(n,p,k) := \frac{{n \choose 3}^kp^{3k}}{k!}\exp\left(-\frac16n^3p^3 + \frac{(1-\theta)^2}{4}n^4p^5 - \frac{(1-\theta)^2(2\theta + 7)}{12}n^5p^7\right)\ .$$
For the computations below, it will be convenient to use the equivalent expanded form of $W$, 
\begin{align}\label{eq:Wdefinition}
   W = \frac{{n\choose 3}^kp^{3k}}{k!} \exp\left(-\frac{n^3p^3}{6}-3knp^2+\frac{n^4p^5}{4}-\frac{9k^2p}{n}+6kn^2p^4-\frac{7n^5p^7}{12}+\frac{9k^2}{pn^2}-\frac{36k^3}{p^2n^4}\right).
\end{align}

Note that $W$ is similar to the expression from \eqref{eq:fixed-Q-P(F)} but is a function of $n, p,$ and $k$ which does not depend on $Q, R,$ or $S$. 
Indeed, we show in the following lemma that we may drop the dependence on $Q, R, S$ in the results from the previous section by showing that their joint distribution is asymptotically that of independent Poisson random variables.

Let 
$$\lambda_Q :=\frac{9k^2}{pn^2}, \quad \lambda_R :=\frac{108k^3}{p^2n^4}, \quad \lambda_S :=\frac{18k^3}{p^2n^4}\ .$$ 
Observe that $\lam_Q = \frac{\theta^2}{4}n^4p^5, \lam_R = \frac{\theta^3}{2}n^5p^7$, and $\lam_S = \frac{\theta^3}{12}n^5p^7$. 
Define $\nu_{p,k,\typ}$ to be the distribution on triangle configurations induced by $G(n,p)$ conditioned on the event $\{\Tri(G) \in \cF_{k,\typ}\}$.

Let $\rho^{\Pois}_{\lam}$ denote the probability mass function of a $\Pois(\lam)$ distributed random variable, and let $\rho^{\Pois}$ be the product Poisson measure on $(Q, R, S) \in \Z_{\geq 0}^3$ where $Q \sim \Pois(\lam_Q), R \sim \Pois(\lam_R)$, and $S \sim \Pois(\lam_S)$.

\begin{lemma}
\label{lemma: typ result}
Let $(Q(F), R(F), S(F))_{\nu_{p,k,\typ}}$ denote the distribution on tuples $\{(Q,R,S) \in \Z_{\geq 0}^3 \mid 2Q+3R+3S \leq k\}$ induced by $\nu_{p,k,\typ}$. Then
$$\| (Q(F), R(F), S(F))_{\nu_{p,k,\typ}}-\rho^{\Pois}\|_\TV=o(1)\ ,$$
and
    \begin{align}
     \Pro\left[\Tri(G)\in\cF_{k,\typ}\right]=\sum_{F\in\cF_{k,\typ}} p^{e_F}\cP(F)
     &=(1+o(1))W.
    \end{align}
\end{lemma}

The proof will proceed as follows. We first define a set of ``good'' tuples $\cI$ and a set of configurations $\cF_{\cI}$ that give rise to tuples in $\cI$. The majority of the proof focuses on showing that $\frac{1}{W} \sum_{F \in \cF_{k, \typ} \setminus \cF_{\cI}} p^{e_F}\cP(F) = o(1)$ by way of the estimates from the previous section for fixed values of $Q, R$, and $S$. For $(Q,R,S) \in \cI$, we show that the total contribution from $\cF_{k,Q,R,S,\typ}$ is asymptotic to $W$ times the corresponding Poisson PMFs. We then sum over $\cI$.

\begin{proof}
First, note that the claim follows immediately from \zcref[S]{lemma: formula} in the case of $k = 0$, so we may assume $k \geq 1$. From \eqref{eq:fixed-Q-P(F)} and \eqref{eq:Wdefinition}, we have
\begin{align}\log \big(\E_{\nufix}&[\cP(F)]\big)
-\log \left(\frac{k!}{{n\choose 3}^kp^{3k}}W\right) \\
&\qquad = -3kp + e_F(p - np^2 + 2n^2p^4) - \frac{9k^2p}{n} + 3knp^2 - 6kn^2p^4\\
&\qquad\quad - \frac{9k^2}{pn^2} + \frac{36k^3}{p^2n^4}+ O(p(Q+R+S))+o(1)\ .
\end{align}
As $\frac{9k^2p}{n} = np^2 \lam_Q$ and $\frac{36k^3}{p^2n^4} = \frac{2\lam_Q^2}{k} - \lam_R - \lam_S$, we may simplify and rewrite this as
\begin{align}
\log \big(&\E_{\nufix}[\cP(F)]\big)-\log \left(\frac{k!}{{n\choose 3}^kp^{3k}}W\right)\\
& = (Q+2R+2S - \lam_Q)(np^2) - \lam_Q + \frac{2\lam_Q^2}{k} - \lam_R - \lam_S + O(p(Q+R+S))+o(1)
\end{align}
where we use the fact that $e_F = 3k - Q - 2R - 2S$.
For notational convenience, we define
\begin{align}
    A := \log \left(\E_{\nufix}\left[\cP(F)\right]\right)-\log \left(\frac{k!}{{n\choose 3}^kp^{3k}}W\right)+\lambda_Q+\lambda_R+\lambda_S\ .
\end{align}

Our aim is to compute $\sum_{F \in \cF_{k,\typ}}p^{e_F}\cP(F) = \sum_{(Q, R, S)} \sum_{F \in \cF_{k, Q, R, S, \typ}} p^{e_F}\cP(F)$. We begin with the inner summation. As $e_F = 3k - Q - 2R - 2S$ is fixed, we may pull it out of the summation and rewrite 
\begin{align*} \sum_{F \in \cF_{k,Q,R,S,\typ}}\cP(F) &= |\cF_{k,Q,R,S,\typ}|\sum_{F \in \cF_{k,Q,R,S,\typ}} \nufix(F) \cP(F) \\
    &= |\cF_{k,Q,R,S,\typ}| \E_{\nufix} [\cP(F)]\ .
\end{align*}
By definition, we have $\frac{1}{W}\E_{\nufix} [\cP(F)] = \frac{k!}{{n \choose 3}^kp^{3k}} \exp(A - \lam_Q - \lam_R - \lam_S)$.

To estimate $|\cF_{k,Q,R,S,\typ}|$, recall from \zcref[S]{lemma: indept-conditional} that, uniformly over admissible $(Q, R, S)$, independently embedded components form a valid configuration of $\cF_{k, Q, R, S, \typ}$ with probability $1-o(1)$. Counting the possible component embeddings and dividing by the number of preimages for each configuration gives

\begin{align}
    |\cF_{k,Q,R,S,\typ}| &\sim {{n\choose 3}\choose k-2Q-3R-3S}6^Q60^R10^S{{n\choose 4}\choose Q}{{n\choose 5}\choose R}{{n\choose 5}\choose S}\ . \label{eq:number-typical}
\end{align}

Let $c:=2Q+3R+3S$. This gives us that
\begin{align}
    &\frac{1}{W}\sum_{F\in\cF_{k,Q,R,S, \typ}} p^{e_F}\cP(F) = \frac{k!}{{n \choose 3}^k}p^{e_F-3k} |\cF_{k,Q,R,S,\typ}|\exp(A - \lam_Q - \lam_R - \lam_S)\\
    &=(1+o(1))\frac{p^{-Q-2R-2S}k!\left(\frac{n^3}{6}\right)^{-c}6^Q60^R10^S}{(k-c)!}\cdot \frac{\left(\frac{n^4}{24}\right)^Q}{Q!} \cdot \frac{\left(\frac{n^5}{120}\right)^R}{R!}\cdot \frac{\left(\frac{n^5}{120}\right)^S}{S!}\exp\left(A-\lambda_Q-\lambda_R-\lambda_S\right)\\
    &= (1+o(1))\frac{k!}{(k-c)!}\left(\frac{\left(\frac{\lam_Q}{k^2}\right)^Qe^{-\lambda_Q}}{Q!}\right)\left(\frac{\left(\frac{\lam_R}{k^3}\right)^Re^{-\lambda_R}}{R!}\right)\left(\frac{\left(\frac{\lam_S}{k^3}\right)^Se^{-\lambda_S}}{S!}\right)\exp\left(A\right)\ .\label{eq:comb-count}
\end{align}
We want to take the sum of the above expression over all possible values of $Q, R, S$. 

If $k = O(1)$, then $\lam_Q = o(1), \lam_R = o(1)$, and $\lam_S = o(1)$, so the claim follows immediately from \eqref{eq:comb-count}. Thus, we may assume $k \to \infty$.

Let $w = \min\left\{\frac{1}{n^{1/3}p}, \frac{\sqrt{k}}{\log n}\right\}$ and let 
$$\cI = \{(Q,R,S) \in \Z_{\geq 0}^3 : 2Q+3R+3S\leq \frac{k}{\log n},\ |Q-\lambda_Q|\leq w,\ R\leq w,\ S\leq w\}\ ,$$ 
Let $\cF_{\cI} = \{F \in \cF_{k, \typ}\ |\ (Q(F), R(F), S(F)) \in \cI\}$. 
We will show that configurations corresponding to tuples outside of $\cI$ have negligible contribution.

Let $$\cB_1 := \left\{F \in \cF_{k,\typ} \mid 2Q(F)+3R(F)+3S(F)> \frac{k}{\log n}\right\}$$ and let $\overline{\cI}_1 = \{(Q, R, S)\ |\ (Q, R, S) = (Q(F), R(F), S(F)) \text{ for some } F \in \cB_1\}$. For $(Q, R, S) \in \overline{\cI}_1$, let $y = \max\{Q, R, S\}$ and note that $y \geq \frac{k}{8 \log n}$. By \eqref{eq:comb-count} and the inequality 
$$\frac{k!}{(k-2Q-3R-3S)!} \leq k^{2Q+3R+3S}\ ,$$
we have
\begin{align}
  \frac{1}{W} \sum_{F\in \cB_1} p^{e_F}\cP(F)
   &\leq (1+o(1)) \sum_{(Q, R, S) \in \barI_1} \rho^{\Pois}_{\lam_Q}(Q) \rho^{\Pois}_{\lam_R}(R) \rho^{\Pois}_{\lam_S}(S) \exp(A)\ .
\end{align}
Since $\frac{\lam_Q\log n}{k} = o(1)$ and $\lam_Q \geq \lam_R, \lam_S$ for sufficiently large $n$, we have $y > \lam_Q \geq \lam_R, \lam_S$; this implies $\rho^{\Pois}_{\lam_Q}(y) \geq \rho^{\Pois}_{\lam_R}(y)$ and $\rho^{\Pois}_{\lam_Q}(y) \geq \rho^{\Pois}_{\lam_S}(y)$. Since each Poisson PMF is bounded above by 1, we thus have
$$\rho^{\Pois}_{\lam_Q}(Q) \rho^{\Pois}_{\lam_R}(R) \rho^{\Pois}_{\lam_S}(S) \leq \rho_{\lam_Q}^{\Pois}(y)\ .$$
From the definition of $A$, 
$$A = (Q+2R+2S - \lam_Q)np^2 + \frac{2\lam_Q^2}{k} + O(p(Q+R+S)) + o(1)\ .$$
Since $Q, R, S \leq y$ and $y \geq \frac{k}{8 \log n}$, we can conclude $A = o(y)$ uniformly over $\overline{\cI}_1$. Therefore,
$$\frac{1}{W} \sum_{F\in \cB_1} p^{e_F}\cP(F) \leq (1+o(1))k^2 \sum_{y \geq \frac{k}{8 \log n}}\rho_{\lam_Q}^{\Pois}(y)\exp(o(y))\ .$$
If $k \geq 8 \log n$, using $\rho_{\lam_Q}^{\Pois}(y) \leq \left(\frac{e\lam_Q}{y}\right)^y$ and $\frac{\lam_Q}{y} \leq \frac{8 \lam_Q \log n}{k} = O(np^2\log n) = o(1)$ following from Assumption~\ref{assumption1}, it follows that
$$\frac{1}{W} \sum_{F\in \cB_1} p^{e_F}\cP(F) = o(1)\ .$$
Else, $k < 8 \log n$ implies $\lam_Q = O\left(\frac{k^{5/3}}{n}\right)$, so $k^2\Pro_{Q \sim \Pois(\lam_Q)}[Q \geq 1] = O(k^2\lam_Q) = o(1)$ and the same conclusion holds.

We next consider tuples such that $c=2Q+3R+3S \leq \frac{k}{\log n}$. From \eqref{eq:comb-count}, the contribution of such tuples can be written as
$$(1+o(1))\rho_{\lam_Q}^{\Pois}(Q)\rho_{\lam_R}^{\Pois}(R)\rho_{\lam_S}^{\Pois}(S)\exp(B)$$
where $B:= A + \log \frac{k!}{k^c(k-c)!}$. We then use the bound
\begin{align*}
\log\frac{k!}{k^c(k-c)!} &=\sum_{j=0}^{c-1}\log\left(1-\frac{j}{k}\right)\\
&\leq-\frac{1}{k}\sum_{j=0}^{c-1}j\\
&=-\frac{c(c-1)}{2k}.
\end{align*}
Since $c \geq 2Q$, it follows that 
$$-\frac{c(c-1)}{2k}\leq-\frac{2Q^2}{k}+\frac{Q}{k}\ .$$

Now let $\cB_2$ be the set of $F \in \cF_{k,\typ} \setminus \cB_1$ (meaning $c \leq \frac{k}{\log n}$) such that $Q(F) \geq \lam_Q - w$ but at least one of $Q(F)>\lambda_Q+w$, $R(F)>w$, or  $S(F)>w$ holds. Let $\barI_2$ be the set of corresponding $(Q,R,S)$. 

Let $x:=Q-\lambda_Q$. We partition $\barI_2$ into three sets,
$$\barI_{2,Q}:=\{(Q,R,S)\in\barI_2:Q>\lambda_Q+w\}\ ,$$
$$\barI_{2,R}:=\{(Q,R,S)\in\barI_2:Q\leq\lambda_Q+w,\ R>w\}\ ,$$
and
$$\barI_{2,S}:=\{(Q,R,S)\in\barI_2:Q\leq\lambda_Q+w,\ R\leq w,\ S>w\}\ .$$

First consider $\barI_{2,Q}$. From the computations above,
$B\leq A-\frac{2Q^2}{k}+\frac{Q}{k}$, 
and hence, since $Q=\lambda_Q+x$,
$$B\leq xnp^2+2(R+S)np^2-\frac{4\lambda_Qx+2x^2}{k}+\frac{Q}{k}+O(p(Q+R+S))+o(1)\ .$$
In particular, the $Q$-dependent positive terms are $O((np^2+p+1/k)x)+o(1)$.

If $0<x\leq\lam_Q$, the standard Poisson bound gives
$$\rho_{\lam_Q}^{\Pois}(Q)\leq\exp\left(-\Theta\left(\frac{x^2}{\lambda_Q}\right)\right)\ .$$
Since $x\geq w$ and $\frac{\lambda_Q(np^2+p+1/k)}{w}=o(1)$, 
the positive $Q$-dependent contribution from $B$ is $o(x^2/\lambda_Q)$. Thus
$$\rho_{\lambda_Q}^{\Pois}(Q)\exp(B)\leq\exp\left(-\Theta\left(\frac{x^2}{\lambda_Q}\right)\right)\exp\left(O((np^2+p)(R+S))\right)\ .$$
If $x>\lambda_Q$, then the Poisson upper-tail bound instead gives
$$\rho_{\lambda_Q}^{\Pois}(Q)\leq\exp(-\Theta(Q))\ ,$$
and since $np^2+p=o(1)$, the $Q$-dependent terms in $B$ are $o(Q)$. Hence
$$\rho_{\lambda_Q}^{\Pois}(Q)\exp(B)\leq\exp(-\Theta(Q))\exp\left(O((np^2+p)(R+S))\right)\ .$$

For the $R$ and $S$ variables we use their Poisson probabilities without discarding the normalization. If $R\leq w$, then
$$(np^2+p)R=o(1)\ .$$
If $R>w$, then $\lambda_R/w=o(1)$, so
$R/\lambda_R\to\infty$ uniformly, and
$$\rho_{\lambda_R}^{\Pois}(R)\exp\left(O((np^2+p)R)\right)\leq\exp\left(-\Theta\left(R\log\frac{R}{\lambda_R}\right)\right)\ .$$
The analogous statement holds for $S$.
Consequently the sums over $R$ and $S$ contribute only a bounded
factor, uniformly in $Q$.

Since $w^2/\lambda_Q\to\infty$, it follows that
$$\sum_{(Q,R,S)\in\barI_{2,Q}}\rho_{\lambda_Q}^{\Pois}(Q)\rho_{\lambda_R}^{\Pois}(R)\rho_{\lambda_S}^{\Pois}(S)\exp(B)=o(1)\ .$$

We next consider $\barI_{2,R}$. Here $-w\leq x\leq w$.
Using $$w(np^2+p)=o(1),\qquad \frac{\lambda_Qw}{k}=o(1),\qquad \frac{w^2}{k}=o(1)\ ,$$
the entire $Q$-dependent contribution to $B$ is $o(1)$. Moreover, $R>w$ and $\lambda_R/w=o(1)$, so
$$\rho_{\lambda_R}^{\Pois}(R)\exp(B)\leq\exp\left(-\Theta\left(R\log\frac{R}{\lambda_R}\right)\right)\exp\left(O((np^2+p)S)\right)\ .$$
As above, summing over $Q$ and $S$ contributes only a bounded
factor, while the Poisson upper tail for $R$ is $o(1)$. Hence
$$\sum_{(Q,R,S)\in\barI_{2,R}}\rho_{\lambda_Q}^{\Pois}(Q)\rho_{\lambda_R}^{\Pois}(R)\rho_{\lambda_S}^{\Pois}(S)\exp(B)=o(1)\ .$$

The case $\barI_{2,S}$ is identical. We therefore conclude that
$$\sum_{(Q,R,S)\in\barI_2}\rho_{\lambda_Q}^{\Pois}(Q)\rho_{\lambda_R}^{\Pois}(R)\rho_{\lambda_S}^{\Pois}(S)\exp(B)=o(1)\ .$$

Lastly, let $\cB_3:=\{F\in\cF_{k,\typ}\setminus\cB_1: Q(F)<\lambda_Q-w\}$, 
and let $\barI_3$ be the set of corresponding $(Q,R,S)$ tuples.
If $\lambda_Q\leq w$, then $\cB_3$ is empty, so suppose
$\lambda_Q>w$. For notational convenience, redefine
$x:=\lambda_Q-Q$. 
Then for every $(Q,R,S)\in\barI_3$ we have
$w<x\leq\lambda_Q$.

Recall that $B\leq A-\frac{2Q^2}{k}+\frac{Q}{k}$. Since $Q=\lambda_Q-x$, the definition of $A$ gives
\begin{align*}
B&\leq -xnp^2+2(R+S)np^2+\frac{2\lambda_Q^2-2Q^2}{k}+\frac{Q}{k}+O(p(Q+R+S))+o(1)\\
&=-xnp^2+2(R+S)np^2+\frac{4\lambda_Qx-2x^2}{k}+\frac{Q}{k}+O(p(Q+R+S))+o(1).
\end{align*}
By the same argument used in the analysis of $\barI_{2,Q}$, the terms involving $R$ and $S$ contribute a bounded total sum.
Moreover, $p\lambda_Q=o(1)$ and $\lambda_Q/k=o(1)$, so
$$B\leq C\left(np^2+p+\frac{\lambda_Q}{k}\right)x+o(1)$$
for some constant $C>0$. By our choice of $w$,
$$\frac{\lambda_Q}{w}\left(np^2+p+\frac{\lambda_Q}{k}\right)=o(1)\ .$$
Since $x\geq w$, it follows from the previous computations that
$$B=o\left(\frac{x^2}{\lambda_Q}\right)+o(1)$$

For $0\leq x\leq\lambda_Q$, the standard lower-tail bound for a
Poisson random variable gives
$$\Pro_{Q \sim \Pois(\lam_Q)}\left[Q\leq\lambda_Q-x\right]\leq\exp\left(-\frac{x^2}{2\lambda_Q}\right).$$
Using our bound on $B$ and summing first over $R$ and $S$, whose Poisson
probabilities have total mass at most $1$, we therefore obtain
$$\sum_{(Q,R,S)\in\barI_3}\rho_{\lambda_Q}^{\Pois}(Q)\rho_{\lambda_R}^{\Pois}(R)\rho_{\lambda_S}^{\Pois}(S)\exp(B) \leq\exp\left(-\Theta\left(\frac{w^2}{\lambda_Q}\right)\right)=o(1)$$
where the last equality follows from $\frac{w^2}{\lambda_Q}\longrightarrow\infty$. 
Thus the total contribution from $\cB_3$ is $o(1)$.

Putting everything together, we conclude that 
$$\frac{1}{W}\sum_{F \in \cF_{k, \typ} \setminus \cF_{\cI}} p^{e_F} \cP(F) = o(1)\ .$$
Moreover, $\rho^{\Pois}\left(2Q + 3R + 3S > \frac{k}{\log n}\right) = O\left(\frac{\lam_Q \log n}{k}\right) = o(1)$ by Markov's inequality. 
Together with $\frac{\lam_Q}{w^2} = o(1)$, $\frac{\lam_R}{w} = o(1)$, $\frac{\lam_S}{w} = o(1)$, this yields $\rho^{\Pois}(\cI^c)=o(1)$. 

For $(Q, R, S)\in \cI$, 
we have that $\log\frac{k!}{k^c(k-c)!}=-\frac{c(c-1)}{2k}+O\left(\frac{c^3}{k^2}\right)$. By the definition of $\cI$ and our previous estimates, $\frac{c}{k} = o(1)$ and $\frac{c^3}{k^2} = o(1)$. Thus
\begin{align}
    B&= A-\frac{c^2}{2k}+o(1)\\
    &=  A-\frac{2Q^2}{k}+o(1)\\
    &= (Q+2R+2S - \lam_Q)(np^2)  + \frac{2\lam_Q^2}{k}-\frac{2Q^2}{k}  + O(p(Q+R+S))+o(1)\\
    &= (Q - \lam_Q)(np^2)  + \frac{2(\lam_Q^2-Q^2)}{k} +o(1)
\end{align}
Indeed, $(R+S)np^2\leq 2wnp^2=o(1)$ and $p(Q+R+S)=O(p\lambda_Q+pw)=o(1)$. Furthermore, $|(Q-\lambda_Q)np^2|\leq wnp^2=o(1)$ and $\frac{|\lambda_Q^2-Q^2|}{k}=\frac{|Q-\lambda_Q|(\lambda_Q+Q)}{k}=O\left(
\frac{\lambda_Qw}{k}+\frac{w^2}{k}\right)=o(1)$. As a result,
$$B = o(1)$$
uniformly over $(Q,R,S) \in \cI$.

It follows from \eqref{eq:comb-count} that, uniformly over $(Q,R,S)\in\cI$,
$$\frac{1}{W}\sum_{F\in\cF_{k,Q,R,S,\typ}}p^{e_F}\cP(F)=(1+o(1))\rho_{\lambda_Q}^{\Pois}(Q)\rho_{\lambda_R}^{\Pois}(R)\rho_{\lambda_S}^{\Pois}(S)\ .$$
We have shown that the contribution of all tuples outside of $\cI$ is $o(W)$. Hence,
\begin{align*}
\sum_{F\in\cF_{k,\typ}}p^{e_F}\cP(F)&=W\sum_{(Q,R,S)\in\cI}(1+o(1))\rho_{\lambda_Q}^{\Pois}(Q)\rho_{\lambda_R}^{\Pois}(R)\rho_{\lambda_S}^{\Pois}(S)+o(W)\\
&=(1+o(1))W,
\end{align*}
This proves the second statement of the claim.

Normalizing by the preceding statement, for every $(Q,R,S)\in\cI$, we have
$$\nu_{p,k,\typ}(Q(F)=Q,R(F)=R,S(F)=S)=(1+o(1))\rho_{\lambda_Q}^{\Pois}(Q)\rho_{\lambda_R}^{\Pois}(R)\rho_{\lambda_S}^{\Pois}(S)$$
uniformly over $\cI$. We have also shown that $\nu_{p,k, \typ}(\cI^c) = o(1)$ and $\rho^{\Pois}(\cI^c) = o(1)$. 
Therefore
\begin{align}
   \|(Q, R, S)_{\nu_{p,k,\typ}} - \rho^{\Pois}\|_{\TV} = o(1) 
    \label{eq:poisson-dist}
\end{align}
proving the first statement.
\end{proof}

\section{Proof of Theorem~\ref{main theorem}}\label{section: final formulas}

We now have everything we need to prove \zcref[S]{main theorem}. Recall that we defined $\mu = {n \choose 3}p^3$.

\begin{proof}[Proof of Theorem~\ref{main theorem}]

We start by showing that the ratio $\Pro[X=k]/\Pro[X\le k]$ is bounded away from $0$ uniformly over the range of $k$. Recall that \zcref[S]{lemma: forbidden subgraphs} tells us that 
$$\Pro[X \leq k] = (1+o(1))\Pro[\Tri(G) \in \cF_{\leq k, \typ}]\ .$$
Thus, it suffices to determine the asymptotics of the right-hand expression.
Viewing $W = W_{n,p}(k)$ as a function of $k$, Lemma \ref{lemma: typ result} gives
\begin{align}
       & \Pro[\Tri(G)\in \cF_{k, \typ}]= (1+o(1)) W(k).
        \label{eq:=k}
\end{align}
By the definition \eqref{eq:Wdefinition}, for every $1 \leq i \leq k$, it holds that
$$\frac{W(i-1)}{W(i)} = \frac{i}{\mu} \exp\left(3np^2 - 6n^2p^4 + (i^2 - (i-1)^2)\left(\frac{9p}{n} - \frac{9}{pn^2}\right) + (i^3 - (i-1)^3)\frac{36}{p^2n^4}\right)\ .$$
By Assumption~\ref{assumption1}, each term in the exponential is $o(1)$, and since we assume $k \leq (1-\epsilon)\mu$ for $\epsilon > 0$, we have
$\frac{W(i-1)}{W(i)} = (1+o(1))\frac{i}{\mu} \leq (1+o(1))\frac{k}{\mu} \leq 1 - \frac{\epsilon}{2}$
which implies
$$W(i) \leq \left(1-\frac{\epsilon}{2}\right)^{k-i}W(k)\ .$$
This is enough to show what we want, as
\begin{align}\Pro[X \leq k] = (1+o(1)) \sum_{i \leq k} \Pro[\Tri(G) \in \cF_{i, \typ}] 
    &= \sum_{i \leq k} (1+o(1))W(i)\\ 
    &\leq \left(\frac{2}{\epsilon}+o(1)\right)W(k)\\
    &\leq \left(\frac{2}{\epsilon}+o(1)\right)\Pro[X=k]\ ,
\end{align}
which tells us 
$$\Pro[X=k \mid X\leq k] = \frac{\Pro[X=k]}{\Pro[X \leq k]} \geq \frac{\epsilon}{2}+o(1) > 0\ .$$
 \newpage
 Thus, $\Pro[\{X = k\} \wedge \{\Tri(G) \in \cF_k \setminus \cF_{k, \typ}\}] \leq \Pro[\{X \leq k\} \wedge \{\Tri(G) \in \cF_{\leq k} \setminus \cF_{\leq k, \typ}\}] = o(\Pro[X \leq k]) = o(\Pro[X=k])$, so
\begin{align}
    \Pro\left[X=k\right]&= (1+o(1)) \Pro\left[\Tri(G)\in\cF_{k,\typ}\right]= (1+o(1))
    W \label{eq: typ_restrict}
\end{align}
as desired.
\end{proof}

\section{Distribution and Sampling}
\label{sec: sampling}
In this section we use the previous results to study the distribution $\nu_{p,k}$ 
of the triangle configuration of $G = G(n,p)$ conditioned on the event $\{X(G) = k\}$, 
which in turn allows us to efficiently sample such graphs.

\subsection{Distribution of Triangles}

In this subsection, we prove \zcref[S]{theorem: dist} and \zcref[S]{theorem: typical-char}. Recall that $\nu_{p,k}$ is the distribution of the triangle configuration of $G(n,p)$ conditioned on having $k$ triangles and $\nu_{p, k}^{\tilt}$ is the distribution of a $k$-triangle configuration $F$ chosen with probability proportional to $p^{e_F}$. 
We write $\nu_{p,k,\elem}^{\tilt}$ and $\nu_{p,k, \typ}^{\tilt}$ for the conditional measures $\nu_{p,k}^{\tilt}(\cdot \mid F \in \cF_{k, \elem})$ and $\nu_{p,k}^{\tilt}(\cdot \mid F \in \cF_{k, \typ})$, respectively.
Our goal is to show that
\begin{align}
    \| \nu_{p,k}- \nu_{p, k, \elem}^{\tilt}\|_{\TV}=o(1) \,.
\end{align} 
We first show that it suffices to compare to the distribution on typical configurations, rather than elementary configurations.
\begin{lemma}\label{lemma:normal-is-typical}
    $$\|\nu_{p,k,\elem}^{\tilt} - \nu_{p,k,\typ}^{\tilt}\|_{\TV} = o(1)\ .$$
\end{lemma}

\begin{proof}

Let
\[
    \cI_k
    :=\bigl\{(Q,R,S)\in \Z_{\geq 0}^3:
    2Q+3R+3S\leq k\bigr\}.
\]
and let 
$$\nu_{p,k,Q,R,S,\mathrm{elem}}^{\tilt} :=\nu_{p,k,\mathrm{elem}}^{\tilt}(\cdot \mid Q(F) = Q, R(F) = R, S(F) = S)\ .$$
For fixed $Q, R, S$, note that $\nu_{p,k,Q,R,S,\mathrm{elem}}^{\tilt}$ is uniform over $\cF_{k,Q,R,S,\mathrm{elem}}$.
Recall the distribution $\tnu = \tnu_{k,Q,R,S}^{\ind}$ defined in Section~\ref{subsec:fixed-component-counts} on embeddings $\mathcal{T}(\boldsymbol{\phi})$ of $Q$ diamonds, $R$ piglets, $S$ 3-books, and $k-2Q-3R-3S$ triangles placed independently and uniformly at random. Let $\cL$ be the event $\{\mathcal T(\boldsymbol{\phi}) \in \cF_{k,Q,R,S,\elem}\}$ that the resulting configuration is a valid elementary configuration, and recall that $\cE$ is the event that the resulting configuration is a valid typical configuration. Then $\cE \subseteq \cL$, and by \zcref[S]{lemma: indept-conditional}, $\tnu(\cE \mid \cL) \geq \tnu(\cE) = 1-o(1)$ uniformly in $Q, R, S$. By the same argument as in \zcref[S]{lemma: indept-conditional}, the distribution $\tnu(\cdot \mid \cL)$ is the same as that of $\nu_{p,k,Q,R,S,\mathrm{elem}}^{\tilt}$, and so
$$\nu_{p,k,Q,R,S,\mathrm{elem}}^{\tilt}(\cF_{k,\typ}) = 1-o(1)$$
uniformly over $(Q, R, S) \in \cI_k$, which implies
$$\nu_{p,k,\elem}^{\tilt}(\cF_{k,\typ}) = 1-o(1)\ .$$
Since $\cF_{k,\typ} \subseteq \cF_{k,\elem}$, we have $\nu_{p,k,\typ}^{\tilt} = \nu_{p,k,\elem}^{\tilt}(\cdot \mid \cF_{k,\typ})$. Therefore,
$$\|\nu_{p,k,\elem}^{\tilt} - \nu_{p,k,\typ}^{\tilt}\|_{\TV} = \nu_{p,k,\elem}^{\tilt}(\cF_{k,\typ}^c) = o(1)\ .$$
\end{proof}

We proceed to proving \zcref[S]{theorem: dist}. To do so, we will use a standard inequality for total variation distance. Suppose $\mu$ and $\mu'$ are measures that can be written as $\mu = \sum_i w_i \mu_i$ and $\mu' = \sum_i w_i' \mu_i'$ for probability vectors $(w_i)_i, (w_i')_i$ and measures $(\mu_i)_i, (\mu_i')_i$. By the triangle inequality and joint convexity of total variation distance,
\begin{equation}\label{eqn:TV-convexity}
    \|\mu - \mu'\|_{\TV} \leq \frac{1}{2}\sum_i |w_i - w_i'| + \sum_i w_i\|\mu_i - \mu_i'\|_{\TV} \ .
\end{equation}
 
\begin{proof}[Proof of Theorem~\ref{theorem: dist}]
By the triangle inequality,
$$\|\nu_{p,k} - \nu_{p,k,\elem}^{\tilt}\|_{\TV} \leq \|\nu_{p,k} - \nu_{p,k,\typ}\|_{\TV} + \|\nu_{p,k,\typ} - \nu_{p,k,\typ}^{\tilt}\|_{\TV}\ + \|\nu_{p,k,\typ}^{\tilt} - \nu_{p,k,\elem}^{\tilt}\|_{\TV}.$$
By the definition of total variation distance, $\|\nu_{p,k} - \nu_{p,k,\typ}\|_{\TV} = \nu_{p,k}(\cF_{k,\typ}^c)$ which is $o(1)$ by the proof of \zcref[S]{main theorem}.
By \zcref[S]{lemma:normal-is-typical}, $\|\nu_{p,k,\typ}^{\tilt} - \nu_{p,k,\elem}^{\tilt}\|_{\TV} = o(1)$ as well. It remains to bound $\|\nu_{p,k,\typ} - \nu_{p,k,\typ}^{\tilt}\|_{\TV}$.

Observe that $\nu_{p,k,\typ}^{\tilt}$ conditioned on $Q, R, S$ is exactly $\nu_{k, Q, R, S,\typ}^{\unif}$, the uniform distribution on $\cF_{k,Q,R,S,\typ}$. Let $\nu_{p,k,Q,R,S,\typ} = \nu_{p,k,\typ}(\cdot \mid Q,R,S)$. 
By \eqref{eqn:TV-convexity}, we then have

\begin{align}
    &\|\nu_{p,k,\typ} - \nu_{p,k,\typ}^{\tilt}\|_{\TV} \leq \|(Q, R, S)_{\nu_{p,k,\typ}} - (Q, R, S)_{\nu_{p,k,\typ}^{\tilt}}\|_{\TV} \\
    &\ \ \ +\sum_{2Q'+3R'+3S'\leq k} \nu_{p,k,\typ}(Q=Q',R=R',S=S') \|\nu_{p,k,Q',R',S',\typ} - \nu_{k,Q',R',S',\typ}^{\unif}\|_{\TV} \label{ineq:tilttyp}
\end{align}
where we recall that $(Q, R, S)_{\nu_{p,k,\typ}}$ denotes the distribution of $(Q(F), R(F), S(F))$ induced by $F \sim \nu_{p,k,\typ}$ and similarly for $(Q, R, S)_{\nu_{p,k,\typ}^{\tilt}}$.
We will show that both terms on the right-hand side are $o(1)$.

For the first term, recall that $\rho^{\Pois}$ denotes the product Poisson distribution over $(Q, R, S)$. By \zcref[S]{lemma: typ result},
$$\|(Q,R,S)_{\nu_{p,k,\typ}} - \rho^{\Pois}\|_{\TV} = o(1)\ .$$
On the other hand, we claim that
\begin{equation}\label{eq:tilttoPoisson}
\|\rho^{\Pois} - (Q,R,S)_{\nu_{p,k,\typ}^{\tilt}}\|_{\TV} = o(1)
\end{equation}
so by the triangle inequality, the first term of \eqref{ineq:tilttyp} is $o(1)$.

To see this, we essentially repeat the arguments from the proof of \zcref[S]{lemma: typ result} for the tilted measure. Indeed, the tilted weight of $|\cF_{k,Q,R,S,\typ}|$ is, by \eqref{eq:number-typical},
$$p^{3k-Q-2R-2S}|\cF_{k,Q,R,S,\typ}| = (1+o(1))C_{n,p,k}\rho_{\lam_Q}^{\Pois}(Q)\rho_{\lam_R}^{\Pois}(R)\rho_{\lam_S}^{\Pois}(S)\exp(\widetilde{B})$$
where $\widetilde{B} = \frac{2\lam_Q^2}{k} + \log \frac{k!}{k^c(k-c)!}$ for $c = 2Q+3R+3S$. Defining the same sets $\cI, \cB_1, \cB_2, \cB_3$ and repeating the same estimates, we can show that $\widetilde{B} = o(1)$ and the product Poisson distribution assigns mass $1-o(1)$ to the set $\cI$. Normalization then provides us with the desired statement \eqref{eq:tilttoPoisson}.
 
It remains to bound the second term of \eqref{ineq:tilttyp}. Fix $Q'$, $R'$, and $S'$ such that $2Q'+3R'+3S'\leq k$, and for the remainder of the proof let 
$$\nu = \nu_{p,k,Q',R',S',\typ},\ \qquad \nu^U = \nu_{k,Q',R',S',\typ}^{\unif},\ \qquad \cF = \cF_{k,Q',R',S',\typ}.$$  
Since $e_F$ is constant over $\cF$, we have that
$$\nu(F) = \frac{\cP(F)}{\sum_{F' \in \cF} \cP(F')}\ .$$
Consequently,
\begin{align*} \|\nu - \nu^U\|_{\TV} &= \frac12 \sum_{F \in \cF} \left| \frac{\cP(F)}{\sum_{F' \in \cF} \cP(F')} - \frac{1}{|\cF|} \right| \\
&= \frac12 \sum_{F \in \cF} \frac{1}{|\cF|} \left|\frac{\cP(F)}{\sum_{F' \in \cF} \frac{1}{|\cF|} \cP(F')} - 1\right| \\
&= \frac12 \E_{\nu^U} \left|\frac{\cP(F)}{\E_{\nu^U}[\cP(F)]} - 1\right|
\end{align*}

Let $\lam = \frac{p}{2} - np^3$. By Lemma \ref{lemma:expectation-exchange},
$$\frac{\cP(F)}{\E_{\nu^U}[\cP(F)]} = (1+o(1))\exp(-\lam (D(F) - \E_{\nu^U}[D(F)]))$$
uniformly over the relevant values of $Q', R', S'$. Let 
$$Y(F) := \exp(-\lam (D(F) - \E_{\nu^U}[D(F)]))\ .$$
Applying Lemma \ref{lemma:expectation-exchange} for both $\lam$ and $2\lam$ gives
$$\E_{\nu^U}[Y] = \exp(\lam\E_{\nu^U}[D(F)])\E_{\nu^U}[\exp(-\lambda D(F))] = 1+o(1),\ \text{and}$$
$$\E_{\nu^U}[Y^2] = \exp(2\lam\E_{\nu^U}[D(F)])\E_{\nu^U}[\exp(-2\lambda D(F))] = 1+o(1)\ .$$
By Cauchy-Schwarz,
$$\E_{\nu^U}|Y - 1| \leq \sqrt{\E_{\nu^U}[(Y-1)^2]} = \sqrt{\E_{\nu^U}[Y^2] - 2\E_{\nu^U}[Y] + 1} = o(1)\ .$$
As a result,
$$\E_{\nu^U} \left|\frac{\cP(F)}{\E_{\nu^U}[\cP(F)]} - 1\right| \leq (1+o(1))\E_{\nu^U}|Y-1| +o(1) = o(1)\ .$$

Thus, $\|\nu - \nu^U\|_{\TV} = o(1)$ uniformly over the relevant values of $Q', R', S'$. The second term in \eqref{ineq:tilttyp} is then
\begin{align*}\sum_{2Q'+3R'+3S' \leq k} &\nu_{p,k,\typ}(Q=Q', R=R', S=S')\|\nu_{p,k,Q',R',S',\typ} - \nu_{k,Q',R',S',\typ}^{\unif}\|_{\TV} \\
&\leq \sup_{2Q'+3R'+3S' \leq k} \|\nu_{p,k,Q',R',S',\typ} - \nu_{k,Q',R',S',\typ}^{\unif}\|_{\TV} \\
&= o(1)
\end{align*}
which along with the bound on the $(Q,R,S)$ marginals finishes the proof.
\end{proof}

From this, we have a description of $\nu_{p,k}$ up to the threshold $n^{-2/3}$ as the tilted measure conditioned on typicality. We now prove \zcref[S]{theorem: typical-char} further characterizing the conditioned measure.

\begin{proof}[Proof of Theorem~\ref{theorem: typical-char}]
    We first prove Part~\ref{theorem: typical-char-part3}. Suppose $p = o(n^{-3/4})$ and $k$ is an integer satisfying $0 \leq k \leq (1-\epsilon) \E_p[X]$ for a fixed $\epsilon > 0$. We want to show that the tilted measure is close to the tilted measure conditioned on being elementary. As in the previous proof, by applying the triangle inequality and \zcref[S]{lemma:normal-is-typical}, it suffices to bound $\|\nu_{p,k,\typ}^{\tilt} - \nu_{p,k}^{\tilt}\|_{\TV}$.
    By definition of total variation distance, $\|\nu_{p,k,\typ}^{\tilt} - \nu_{p,k}^{\tilt}\|_{\TV}$ is exactly $\nu_{p,k}^{\tilt}(\cF_{k, \typ}^c)$, so it suffices to show that atypical triangle configurations are unlikely under the tilted measure.

    Given $F \in \cF_k$, let $N_v(F)$ be the number of components in $\dual{G_F}$ whose underlying graph in $G_F$ has exactly $v$ vertices (so, for example, diamond and $K_4$ components of $F$ are both counted by $N_4(F)$). We will show that $\sum_{v \geq 5} \E_{\nu_{p,k}^{\tilt}}[N_v] = o(1)$ and then separately analyze the edge-disjoint atypical components.

    Define the partition function 
    $$Z_k = Z_{p,k}^{\tilt} := \sum_{F \in \cF_k} p^{e_F}\ .$$
    Then
    $$\E_{\nu_{p,k}^{\tilt}}[N_v] = \frac{1}{Z_k}\sum_{F \in \cF_k} N_v(F)p^{e_F}\ .$$
    We may rewrite this as a sum over pairs $(F, C)$ where $F \in \cF_k$ and $C$ is a component of $\dual{G_F}$ corresponding to a subgraph $H_C$ of $G_F$ such that $v(H_C) = v$. Let $\cH_v$ be the set of labeled subgraphs $H$ on $v$ vertices that occur as the underlying graph of a connected component in $\dual{G_F}$. Deleting $\Tri(H_C)$ from $F$ results in a triangle configuration $F'$ with $k - |\Tri(H_C)|$ triangles, so we can define an injection from pairs $(F,C)$ to pairs $(F', H)$ for $H \in \cH_v$. 
    Moreover, $p^{e_F}$ factorizes as $p^{e(H)}p^{e_{F'}}$.
    \begin{align*}
        \E_{\nu_{p,k}^{\tilt}}[N_v] &= \frac{1}{Z_k}\sum_{(F,C)} p^{e_F}\\
        &\leq \frac{1}{Z_k} \sum_{(F', H)} p^{e(H)}p^{e_{F'}}\\
        &= \sum_{H \in \cH_v} p^{e(H)} \frac{Z_{k - |\Tri(H)|}}{Z_k}\ .
    \end{align*}
We bound the ratio of partition functions using a straightforward counting argument. Given a configuration with $k-t$ triangles, we can obtain a configuration with $k$ triangles by adding $t$ new mutually disjoint triangles. At each step, there are at least ${n-3k \choose 3} = \left(1 - O\left(\frac{k}{n}\right)\right){n \choose 3}$ choices of vertices for the next triangle, and each resulting configuration arises from at most $t!{k \choose t}$ such placements. 
\begin{align*}
    Z_k &\geq \frac{{n-3k \choose 3}^t}{t!{k \choose t}} p^{3t} Z_{k-t} \geq \left(\frac{(1-O(\frac{k}{n})){n \choose 3}p^3}{k}\right)^tZ_{k-t} \,.
\end{align*}
As $k \leq (1-\epsilon)\mu$ and $\frac{k}{n} = o(n^{-1/4})$, this gives us
$$\frac{Z_{k-t}}{Z_k} \leq (1-c(\epsilon))^t\ .$$
for some small constant $c(\epsilon) > 0$.

Lastly, we need to bound the size of $\cH_v$. Given $H \in \cH_v$, observe that there exists an ordering of the vertices of $H$, say $x_1, x_2, \dots, x_v$, such that we may build $H$ in the following way:
\begin{enumerate}
    \item $H_3$ is the triangle $x_1x_2x_3$ and is contained in $H$.
    \item For each $i > 3$, choose $a(i) < b(i) < i$ such that $x_ix_{a(i)}x_{b(i)}$ forms a triangle in $H$ and $x_{a(i)}x_{b(i)} \in E(H_{i-1})$.
    \item Let $H_i$ be obtained from $H_{i-1}$ by adding the vertex $x_i$ and the edges $x_ix_{a(i)}, x_ix_{b(i)}$, followed by all other edges $x_ix_j \in E(H)$ for $j < i$.
\end{enumerate}
By construction, $H_v = H$. It is clear that this is possible by looking at the triangle-dual graph: as $\dual{H}$ is connected, we may take a rooted spanning tree of $\dual{H}$ and expose it one node at a time so that each parent node is exposed before its children. Starting with the root triangle, we can order the vertices of $H$ in the order that they are exposed. Since every non-root triangle shares an edge with its parent, at most one new vertex is exposed at each step. 
For $\ell > 3$, let $T = x_ix_jx_{\ell} \in \Tri(H)$ where $i, j < \ell$ be the first exposed triangle containing $x_{\ell}$. Because its parent $T'$ was exposed before $T$, $x_{\ell} \notin V(H_{\ell-1})$, but $T$ and $T'$ share an edge, so $x_ix_j \in E(H_{\ell-1})$, as required for the second step of the construction. 

To bound the number of labeled subgraphs $H \in \cH_v$, we encode their unlabeled structure as an auxiliary rooted forest: let the edges $x_1x_2, x_1x_3, x_2x_3$ be three separate roots. For each $i > 3$,
\begin{enumerate}
    \item add a node labeled $x_i$ as the child of node $x_{a(i)}x_{b(i)}$;
    \item add nodes labeled $x_ix_{a(i)}$ and $x_ix_{b(i)}$ as the first and second children of $x_i$, respectively;
    \item for any subsequent edge of the form $x_ix_j$ for $j < i$, add a node labeled $x_ix_j$ as a child of $x_i$.
\end{enumerate}
The result is a rooted forest of order $e(H) + (v-3)$. By maintaining an ordering on the children of each node, we may also ensure that we can distinguish between nodes added in step 2 and those added in step 3. Letting $m = e(H) - (2v-3)$ be the total number of nodes added in step 3, the number of unlabeled rooted forests possible is at most $C^{v+m}$ for some $C > 0$. The number of ways to label the vertex nodes is at most $n^v$, and for each of the $m$ edge nodes added in step 3, there are at most $v$ choices for its label given its parent. Overall, this tells us there are at most $C^{v+m}n^vv^m$ elements of $\cH_v$ for a given value of $m$. Putting everything together, we obtain the bound

\begin{align*} \sum_{H \in \cH_v} p^{e(H)} &\leq \sum_{m \geq 0} C^{v+m}n^vv^m p^{2v-3+m}\\
&\leq (Cnp^2)^vp^{-3}\sum_{m \geq 0} (Cvp)^m\ .
\end{align*}
As $v \leq k+2$, we know that
$$pv \leq pk \leq (1-\epsilon)n^3p^4$$
which is $o(1)$ precisely when $p = o(n^{-3/4})$. For $n$ large enough, $Cvp \leq \frac12$ so $\sum_{m \geq 0} (Cvp)^m \leq 2$. Then
\begin{align*} \sum_{v \geq 5} \E_{\nu_{p,k}^{\tilt}}[N_v] \leq 2p^{-3}\sum_{v \geq 5} (Cnp^2)^v = O(n^5p^7) = o(1)\ .
\end{align*} 

The same argument does not extend to $v \geq 4$; indeed, $\E_{\nu_{p,k}^{\tilt}}[N_4] \leq C' n^4 p^5$ which is not $o(1)$ when $p \gg n^{-4/5}$. However, we can separately show that $K_4$ components are unlikely, by simply bounding
$$\E_{\nu_{p,k}^{\tilt}}[N_{K_4}] \leq {n \choose 4}p^6 \frac{Z_{k-4}}{Z_k} = o(1)\ .$$
Thus, by Markov's inequality and a union bound, whp under the tilted measure every component of $\dual{G_F}$ corresponds to a diamond or isolated triangle. 

To finish, we must show that the remaining atypical subgraphs whp do not appear. Recall that these consist of a vertex incident to $r = \lceil \frac{\log n}{6} \rceil$ edge-disjoint triangles (type \ref{atypicaltype3}) and a 4-cycle of vertex-connected triangles (type \ref{atypicaltype4}). Let $N_{\text{type 3}}(F)$ and $N_{\text{type 4}}$ be the number of copies of these subgraphs, respectively.

Let $A$ be the event that all triangle components of $G_F$ are triangles or diamonds. Let $\nu = \nu_{p,k}^{\tilt}(\cdot \mid A, Q(F) = Q)$ be the tilted distribution conditioned on having $Q$ diamonds, and let $\nu^{\ind}_{k,Q}$ be the product distribution that results from independently sampling $Q$ uniformly random diamonds and $k-2Q$ uniformly random triangles. By fixing $Q$, the measure $\nu$ is uniform on $\cF_{k,Q,0,0}$ and thus is equal to $\nu^{\ind}_{k,Q}$ conditioned on the event $B$ that the sampled components form a valid configuration in $\cF_{k,Q,0,0}$. By union bound, $\nu^{\ind}_{k,Q}(B^c) = O\left(\frac{k^2}{n^2} + \frac{k^3}{n^3}\right) = o(1)$ uniformly in $Q$, so any event with probability $o(1)$ under $\nu^{\ind}_{k,Q}$ also has probability $o(1)$ under $\nu$; thus, it suffices to show the remaining atypical subgraphs have $o(1)$ probability under $\nu^{\ind}_{k,Q}$.

As each triangle component has at most 4 vertices, the probability that a fixed set of $r$ components share a given vertex is at most $\left(\frac{4}{n}\right)^r$. Taking a union bound over all choices of $r$ components and all vertices, we have
$$\nu^{\ind}_{k,Q}(N_{\text{type 3}} > 0) \leq n{k \choose r}\left(\frac{4}{n}\right)^r = o(1)$$
by the definition of $r$ and since $k = o(n \log n)$.

Any subgraph of Type 4 must consist of two diamonds, a diamond and two triangles, or four triangles sharing vertices to form a $C_4$ triangle-intersection graph. By a similar computation, we have
$$\nu^{\ind}_{k,Q}(N_{\text{type }4} > 0) = O\left(\frac{k^2}{n^2} + \frac{k^3}{n^3} + \frac{k^4}{n^4}\right) = o(1)\ .$$
This proves \zcref[S]{theorem: typical-char}, Part~\ref{theorem: typical-char-part3}. 

Now to prove Part~\ref{theorem: typical-char-part1}, assume $p = o(n^{-4/5})$. We compare the tilted measure with the uniform measure. Recall from the proof of \zcref[S]{lemma: indept-conditional} that the distribution of $k$ independently embedded triangles conditioned on forming a valid triangle configuration gives exactly $\nu_k^{\unif}$. Thus,
$$\nu_{k}^{\unif}(Q = R = S = 0) = 1-o(1)\ .$$
We also know from the preceding argument that for $p = o(n^{-4/5})$, we have $\E_{\nu_{p,k}^{\tilt}}[N_4] = O(n^4p^5) = o(1)$ and $\sum_{v \geq 5} \E_{\nu_{p,k}^{\tilt}}[N_v] = o(1)$. Thus, with probability $1 - o(1)$ under the tilted measure, triangle components consist of isolated triangles. Conditioned on this event, $e_F = 3k$ is constant, so the tilted and uniform laws are identical. This proves $\|\nu_{p,k}^{\tilt} - \nu_k^{\unif}\|_{\TV} = o(1)$.  

For Part~\ref{theorem: typical-char-part2}, let $n^{-4/5} \ll p \ll n^{-3/4}$ and $k = \Theta(\mu)$. We look at the tilted and uniform measures on the event $\{Q = 0\}$. From the previous paragraph, $\nu_k^{\unif}(Q = 0) = 1-o(1)$. On the other hand, let $Z_{Q = q, \typ} = \sum_{F \in \cF_{k,q,0,0, \typ}} p^{e_F}$. Then 
$$\frac{Z_{Q=1, \typ}}{Z_{Q=0, \typ}} \sim \frac{1}{p}6{n \choose 4}\frac{{{n \choose 3} \choose k-2}}{{{n \choose 3} \choose k}} \sim \frac{9k^2}{pn^2} = \lam_Q\ .$$
Since $k =\Theta(\mu)$, we have $\lam_Q = \Theta(n^4p^5) \to \infty$ and thus
$$\nu_{p,k,\typ}^{\tilt}(Q = 0) \leq \frac{Z_{Q=0,\typ}}{Z_{Q=0,\typ}+Z_{Q=1,\typ}} = o(1)\ .$$
It follows from Part~\ref{theorem: typical-char-part3} and \zcref[S]{lemma:normal-is-typical} that $\|\nu_{p,k,\typ}^{\tilt} - \nu_{p,k}^{\tilt}\|_{\TV} = o(1)$, so $\{Q = 0\}$ witnesses $\|\nu_{p,k}^{\tilt} - \nu_k^{\unif}\|_{\TV} = 1-o(1)$. 

Finally, for Part~\ref{theorem: typical-char-part4}, we construct a family of graphs $\cB  \subset \cF_{k,\elem}^c$ with nontrivial tilted measure. Fix vertices $u, v$ and let $uv$ form an edge. Then choose a set $L \subset V \setminus \{u, v\}$ of $\ell$ vertices (where $\ell$ is to be determined), all of which are adjacent to both $u$ and $v$. Finally, take an equipartition $L_1 \cup L_2$ of $L$ in some canonical way (e.g. let the $\lceil \frac{\ell}{2}\rceil$ lowest-indexed vertices be in $L_1$ and the remainder in $L_2$) and place $\frac{k-\ell}{2}$ additional edges between $L_1$ and $L_2$. Here we assume $\ell$ is chosen to have the same parity as $k$.

The resulting graph has exactly $k$ triangles: $\ell$ of the form $uvx$ for $x \in L$, $\frac{k-\ell}{2}$ of the form $uxy$ for $x \in L_1, y \in L_2$, and $\frac{k-\ell}{2}$ of the form $vxy$ for $x \in L_1, y \in L_2$. Let $\cB$ be the collection of all $k$-triangle configurations obtained in this way, and let $Z_{\cB} = \sum_{F \in \cB}p^{e_F}$.

First, observe that the triangle-dual graph of any element in $\cB$ consists of a single large connected component, so $\cB \subseteq \cF_{k,\elem}^c$ as desired. There are ${n-2 \choose \ell}{\lfloor\ell^2/4\rfloor \choose (k-\ell)/2}$ choices for $L$ and the edges within it, and each graph has weight $p^{1+2\ell + (k-\ell)/2} = p^{1 + 3\ell/2 + k/2}$. Assuming $\ell = o(k)$ and $k = o(\ell^2)$, we apply Stirling's formula to get
$$\log Z_{\cB} = \frac{k}{2}\log \frac{\ell^2p}{k} + o(k \log n)$$

Now set $p = n^{-\alpha}$ for $\frac23 < \alpha < \frac34$ and $k =\Theta(\mu)$. Then $\mu = \Theta(n^{3-3\alpha})$. Additionally set $\ell = n^{\beta+o(1)}$ for $0 < \beta < 1$. Our requirements of $\ell = o(k)$ and $k = o(\ell^2)$ hold when $\frac{3-3\alpha}{2} < \beta < 3-3\alpha$. Substituting into the above expression, we obtain
$$\log Z_{\cB} = \frac{k}{2}((2\beta - \alpha) - (3-3\alpha))\log n + O(k) +  o(k \log n) = (\beta - \frac{3}{2} + \alpha + o(1))k \log n\ .$$
Thus, $Z_{\cB}$ is exponentially large when $\beta > \frac{3}{2} - \alpha$ (which implies $\beta > \frac{3-3\alpha}{2}$). Observe that the interval $(\frac32 - \alpha, 3-3\alpha)$ is nonempty precisely when $\alpha < \frac{3}{4}$, which is our claimed threshold.

We compare $Z_{\cB}$ to $Z_{\typ} = \sum_{F \in \cF_{k,\typ}}p^{e_F}$, which we may write as 
$$Z_{\typ} = \sum_{2Q + 3R + 3S \leq k} p^{3k-Q-2R-2S} |\cF_{k,Q,R,S,\typ}|\ .$$
Recall \eqref{eq:number-typical} tells us that
$$|\cF_{k,Q,R,S,\typ}| = (1+o(1)){{n\choose 3}\choose k-2Q-3R-3S}6^Q60^R10^S{{n\choose 4}\choose Q}{{n\choose 5}\choose R}{{n\choose 5}\choose S}\ .$$
By a computation similar to that of \eqref{eq:comb-count},
$$p^{3k-Q-2R-2S}|\cF_{k,Q,R,S,\typ}| \leq (1+o(1))\frac{\mu^k}{k!}\frac{\lam_Q^Q}{Q!}\frac{\lam_R^R}{R!}\frac{\lam_S^S}{S!}\ .$$
Summing over all triples gives
$$Z_{\typ} \leq (1+o(1))\frac{\mu^k}{k!} \exp(\lam_Q + \lam_R + \lam_S)\ .$$
Our assumption $k = \Theta(\mu)$ implies $\log \frac{\mu^k}{k!} = O(k)$ and $\lam_Q + \lam_R + \lam_S = o(k)$. Thus 
$$Z_{\typ} \leq \exp(o(k \log n))\ .$$
Then
$$\nu_{p,k}^{\tilt}(\cF_{k,\typ}) = \frac{Z_{\typ}}{Z_k} \leq \frac{Z_{\typ}}{Z_{\cB}} = o(1)\ .$$
\zcref[S]{lemma:normal-is-typical} implies $\nu_{p,k}^{\tilt}(\cF_{k,\typ}) = (1-o(1))\nu_{p,k}^{\tilt}(\cF_{k,\elem})$; hence,
$\nu_{p,k}^{\tilt}(\cF_{k, \elem}^c) \geq 1-o(1)$
as claimed. \end{proof}

\subsection{Algorithm and Distribution of the Non-Triangle Edges}
Now we move on to proving \zcref[S]{theorem: sample}, which states that we can sample from a distribution $\pi_{\alg}$ in $O(n^3\log n)$ time such that 
    \begin{align}
        \| \pi_{\alg}-\pi_{p,k}\|_{\TV}=o(1).
    \end{align}
where we recall that $\pi_{p,k}$ is the distribution of $G(n,p)$ conditioned on having $k$ triangles.

 Our algorithm has two steps. First, we will sample the triangles in the graph. Then, we will sample the remaining edges using a Markov chain, which we will analyze with path coupling and burn-in. 

Fix a sufficiently large constant $T > 0$. The sampling algorithm is given below.
\begin{algorithm}[H]
\caption{Sampling Algorithm}\label{alg:cap}
\begin{algorithmic}
\State $G\gets $ empty graph on a set $V$ of $n$ vertices
\State $Q \gets \mathrm{Pois}(\lambda_Q) $
\State $R \gets \mathrm{Pois}(\lambda_R)$
\State $S \gets \mathrm{Pois}(\lambda_S)$
\If {$2Q+3R+3S>k$}
\State Let $G_0$ be the union of $k$ vertex-disjoint triangles.
\State $G\gets G_0$.
\Else
\State Independently and uniformly embed $k-2Q-3R-3S$ triangles, $Q$ diamonds, $R$ piglets, and $S$ 3-books on $[n]$, and let $G$ be the union of their edges.
\If {$|\Tri(G)| \neq k$}
\State Let $G_0$ be the union of $k$ vertex-disjoint triangles.
\State $G\gets G_0$.
\EndIf
\EndIf
\State Let $F = \Tri(G)$.
\For{$Tn^2\log n$ steps} Glauber dynamics on ${[n] \choose 2} \setminus E(G_F)$
\State choose $e \in {[n] \choose 2} \setminus E(G_F)$ uniformly at random
\If {$G\cup e$ would not form a new triangle}
\State $G \gets G \cup e$ with probability $p$
\State $G \gets G \setminus e$ with probability $1-p$
\EndIf
\EndFor
\end{algorithmic}
\end{algorithm}
\begin{proof}[Proof of Theorem~\ref{theorem: sample}]
Let $\nu_{\alg}$ be the distribution over $F \in \cF_k$ induced by the first part of $\pi_{\alg}$. We compare $\nu_{\alg}$ to $\nu_{p,k}$. First, we can see that $\rho^{\Pois}(2Q + 3R + 3S \leq k) = 1-o(1)$. Conditioned on any such valid $(Q, R, S)$, \zcref[S]{lemma: indept-conditional} shows that the uniformly embedded components form a configuration in $\cF_{k,Q,R,S,\typ}$ with probability $1-o(1)$, uniformly over $(Q, R, S)$, and conditioned on this event, the resulting configuration is distributed as $\nu_{k,Q,R,S,\typ}^{\unif}$. 

On the other hand, recall that $\nu_{p,k,\typ}^{\tilt}$ conditioned on $Q,R,S$ is exactly $\nu_{k,Q,R,S,\typ}^{\unif}$, and by \eqref{eq:tilttoPoisson}, the marginal distribution of $\nu_{p,k,\typ}^{\tilt}$ on $(Q,R,S)$ is within $o(1)$ total variation distance of the product Poisson law. The outcomes in Algorithm~\ref{alg:cap} resulting in either inadmissible triples $(Q,R,S)$ or embeddings that do not form a configuration in $\cF_{k,Q,R,S,\typ}$ have total measure $o(1)$. Hence,
$$\|\nu_{\alg} - \nu_{p,k,\typ}^{\tilt}\|_{\TV} = o(1)\ .$$
By the triangle inequality, \zcref[S]{theorem: dist}, and \zcref[S]{lemma:normal-is-typical}, it follows that
$$\|\nu_{\alg} - \nu_{p,k}\|_{\TV} = o(1)\ .$$

It remains to show that the distribution of the non-$\Tri(G)$ edges of $G \sim G(n,p)$ is close in total variation distance to the output of the Glauber dynamics, and that the algorithm runs in $O(n^3 \log n)$ time. 
Indeed, by the convexity of total variation distance and the triangle inequality, 
we then have that,
$$\|\pi_{\alg} - \pi_{p,k}\|_{\TV} \leq \|\nu_{\alg} - \nu_{p,k}\|_{\TV} + \sum_{F \in \cF_k} \nu_{\alg}(F)\|P^t_F - \pi_F\|_{\TV}$$
where $P^t_F$ denotes the output distribution on $\Omega_F := \{H : E(H) \supseteq E(G_F), \Tri(H) = F\}$ from running Glauber dynamics for $t$ steps starting from $E(G_F)$, and $\pi_F$ denotes the distribution on $\Omega_F$ where $G$ is sampled from $G(n,p)$ conditioned on $\Tri(G) = F$.

Fix $F \in \cF_{k, \typ}$. It suffices to restrict to $\cF_{k,\typ}$ since $\nu_{p,k}(\cF_{k,\typ}^c) = o(1)$ by the proof of \zcref[S]{main theorem} and $\nu_{\alg}(\cF_{k,\typ}^c) = o(1)$ by the preceding argument.

Our approach follows the path coupling method with burn-in for sampling triangle-free graphs in \cite{jenssen2024lower}, adapted to our setting.
To compare $P_F^t$ with $\pi_F$, we first compare $\pi_F$ with the unconditioned measure.
Let $G \sim \pi_F$. Let $B\subseteq \binom{[n]}{2}\setminus E(G_F)$ be obtained by including each edge independently with probability $p$, and define
\begin{align}
G' := \bigl([n],\, E(G_F)\cup B\bigr).
\end{align}
Let $\pi_F'$ be the resulting measure on such graphs $G'$. Although $\Tri(G') \supset F$, the graph $G'$ may contain additional triangles. We may thus couple $G$ and $G'$; indeed, $\pi_F$ is equivalent to $\pi_F'$ conditioned on the decreasing event of creating no new triangles, so by the FKG inequality, there is a coupling such that $G \subseteq G'$ with probability 1 and $\pi_F$ is stochastically dominated by $\pi_F'$. As a result, taking $D = C_0(np + \log n)$ for a fixed sufficiently large constant $C_0 > 0$, we have
$$\pi_F(\Delta(G) \geq D) \leq \pi_F'(\Delta(G') \geq D)\ .$$
Recall that $\Delta(G_F) \leq \log n$ for $F \in \cF_{k, \typ}$, so $\deg_{G'}(v) \leq \log n + \Bin(n,p)$ for $v \in V(G')$.
By Chernoff and union bound, we get that $\pi'_F(\Delta(G') \geq D) \leq n^{-10}$.

Now let $X_0 \in \Omega_F$ be arbitrary and let $(X_t)_{t\geq0}$ be a run of the Glauber dynamics with stationary distribution $\pi_F$.
We claim that there exists $C_1 > 0$ such that for any $\eta > 0$, if $t \geq C_1n^2 \log(n/\eta)$, then 
\begin{align}
    \Pro[\Delta(X_t) \geq D]
    \leq n^{-10} + \eta.
\end{align}

To prove this, let $\Omega_F' := \{H : E(H) \supseteq E(G_F)\}$. Since $\Omega_F \subseteq \Omega_F'$, define a Markov chain $(Y_t)_{t \geq 0}$ on $\Omega_F'$ as follows: let $Y_0 \in \Omega_F'$ be arbitrary. To obtain $Y_t$ from $Y_{t-1}$, choose one of the $\gamma := {n \choose 2} - e_F$ pairs not in $E(G_F)$ uniformly at random and resample it by including it in $Y_t$ with probability $p$. 

Define a coupling $(X_t, Y_t)_{t \geq 0}$ as follows: set $X_0 = Y_0$. At each step, choose the same pair $e \in {[n] \choose 2} \setminus E(G_F)$ , and use the same $p$-biased coin to determine the update in both chains. The result is that $X_t \subseteq Y_t$ for all $t \geq 0$.  

For the remainder of the proof, we use $\Pro$ and $\E$ to denote probabilities and expectations with respect to the coupling under discussion. 

Since the stationary distribution of $(Y_t)_{t \geq 0}$ is precisely $\pi_F'$, the total variation distance between the law of $Y_t$ and $\pi_F'$ is at most the probability that some edge has not yet been updated, so
$$\|\mathrm{Law}(Y_t) - \pi_F'\|_{\TV} \leq \gamma\left(1- \frac{1}{\gamma}\right)^t \leq \gamma e^{-t/\gamma}\ .$$
By taking $t \geq \gamma \log(\gamma/\eta)$, we have
$$\|\mathrm{Law}(Y_t) - \pi_F'\|_{\TV} \leq \eta\ .$$
Set $\eta = n^{-10}$. Then for $t \geq C_1n^2\log n$ where $C_1 > 0$ is some fixed constant,
\begin{align}
    \Pro[\Delta(X_t) \geq D] &\leq \Pro[\Delta(Y_t) \geq D]\\
    &\leq \pi_F'(\Delta(G') \geq D) + \|\mathrm{Law}(Y_t) - \pi_F'\|_{\TV}\\
    &\leq 2n^{-10}\ . \label{eqn:burn-in-prob}
\end{align}

This result allows us to restrict our attention to the bounded-degree regime 
\begin{align}
\Omega_D:=\left\{G \in \Omega_F:\Delta(G)\le D\right\}.
\end{align}

Let $d(\cdot,\cdot)$ denote Hamming distance on the set of non-fixed edges $\binom{[n]}{2}\setminus E(G_F)$. Let $(Z_t)_{t \geq 0}$ be another copy of the Glauber dynamics where $Z_0 \sim \pi_F$ (so in particular, $Z_t \sim \pi_F$ for all $t \geq 0$). Given $(X_t, Z_t)$ where $t \geq C_1n^2 \log n$, we suppose first that $d(X_t, Z_t) = 1$. Couple one step of the Glauber dynamics by choosing an edge $\phi\in \binom{[n]}{2}\setminus E(G_F)$
uniformly at random and attempting to update it in both graphs by maximally coupling the two conditional distributions for the updated status of $\phi$. Indeed, if $X_t$ and $Z_t$ differ at the edge $\{i, j\}$, then we may couple the updates so that $X_{t+1} = Z_{t+1}$ as long as $\phi$ does not complete a triangle in exactly one of the graphs. The latter occurs if and only if either $\phi = \{i, v\}$ with $\{j, v\} \in E(X_t \cap Z_t)$ or $\phi = \{j,v\}$ with $\{i,v\} \in E(X_t \cap Z_t)$, for some $v \in V \setminus \{i,j\}$. Since $X_t,Z_t\in \Omega_D$, there are at most $2D$ such edges.

For each such edge, the two conditional Bernoulli update laws differ by at most $p$, so under a maximal coupling the updated values disagree with probability at most $p$. Therefore
$$\E\left[d(X_{t+1},Z_{t+1}) \,\middle|\, X_t,Z_t\in \Omega_D,\ d(X_t,Z_t)=1\right]
\le 1-\frac{1}{\gamma}+\frac{2Dp}{\gamma} = 1-\frac{1-2Dp}{\gamma}.$$

Writing $\delta:=1-2Dp$, we obtain
\begin{align}
\E\!\left[d(X_{t+1},Z_{t+1}) \,\middle|\, X_t,Z_t\in \Omega_D,\ d(X_t,Z_t)=1\right]
\le \left(1-\frac{\delta}{\gamma}\right).
\label{eq:path-coupling-local}
\end{align}

Now suppose that $d(X_t, Z_t) > 1$. To pass from pairs at distance $1$ to arbitrary pairs in $\Omega_D$, note that if $X, Z\in \Omega_D$ then there is a path
\begin{align}
X=W_0,W_1,\dots,W_{d(X,Z)}=Z
\end{align}
inside $\Omega_D$ with $d(W_{i},W_{i+1})=1$ for every $i \geq 0$. Indeed, first delete the edges of $X\setminus Z$ one at a time, and then add the edges of $Z\setminus X$ one at a time. Deleting edges cannot create new triangles or increase degrees, and once all edges of $X\setminus Z$ have been removed, the remaining edges comprise $E(X \cap Z)$; after that, each added edge belongs to $Z$, so every intermediate graph is a subgraph of $Z$ and hence still lies in $\Omega_D$. Thus the path-coupling theorem of Bubley and Dyer \cite{bubley1997path} applies, and \eqref{eq:path-coupling-local} implies that for all $X,Z \in \Omega_D$,
\begin{align}
\E\!\left[d(X_{t+1},Z_{t+1}) \,\middle|\, X_t=X,\ Z_t=Z,\ X,Z\in \Omega_D\right]
\le \left(1-\frac{\delta}{\gamma}\right)d(X,Z).
\label{eq:path-coupling-global}
\end{align}

Let $\cE_t$ be the event $\{X_t, Z_t \in \Omega_D\}$. Let $d_t = d(X_t, Z_t)$. By linearity of expectation, $$\E[d_{t+1}] = \E[d_{t+1} \mathds{1}_{\cE_{t}}] + \E[d_{t+1} \mathds{1}_{\cE_{t}^c}]\ .$$ 
Since $d(X, Z) \leq \gamma$ for all $X, Z \in \Omega_F$, we have for the second term
$$\E[d_{t+1} \mathds{1}_{\cE_{t}^c}] \leq \gamma \Pro[\cE_t^c] = O(n^{-8})$$
where we use $\gamma \leq {n \choose 2}$, \eqref{eqn:burn-in-prob} for $X_t$, and stationarity of $Z_t$ to obtain the final expression. For the first term, applying Law of Total Expectation gives
\begin{align*} 
    \E[d_{t+1} \mathds{1}_{\cE_{t}}] &= \E[\E[d_{t+1} \mid X_t, Z_t] \mathds{1}_{\cE_{t}}] \\
    &\leq \left(1-\frac{\delta}{\gamma}\right)\E[d_t\mathds{1}_{\cE_t}] \\
    &\leq \left(1-\frac{\delta}{\gamma}\right)\E[d_t]\ .
\end{align*} 
Thus, for every $t_0 \geq C_1n^2\log n$, 
$$\E[d_{t_0+1}] \leq \left(1-\frac{\delta}{\gamma}\right)\E[d_{t_0}] + O(n^{-8})\ .$$
Iterating $t$ times, we obtain the bound
\begin{align*} 
    \E[d_{t_0 + t}] &\leq \left(1-\frac{\delta}{\gamma}\right)^t \E[d_{t_0}] + O(n^{-8})\sum_{i=0}^{t-1}\left(1-\frac{\delta}{\gamma}\right)^i\\
    &\leq \left(1-\frac{\delta}{\gamma}\right)^t \gamma + O(tn^{-8})\\
    &\leq n^2 \exp\left(-\frac{\delta t}{\gamma}\right) + O(tn^{-8})
\end{align*}
Since $\delta = 1-o(1)$ and $\gamma = O(n^2)$, taking $t_0 = C_1n^2 \log n$ and $t = C_2n^2 \log n$ for a sufficiently large constant $C_2 > 0$, we have
$$\Pro[X_{t_0 + t} \neq Z_{t_0+t}] \leq \E[d_{t_0+t}] \leq n^2 \exp\left(-\frac{\delta t}{\gamma}\right) + O(tn^{-8}) = o(1)\ .$$

Thus, choosing the constant $T$ in Algorithm~\ref{alg:cap} so that $T \geq C_1 + C_2$ gives
$$\sup_{F \in \cF_{k,\typ}} \|P_F^{Tn^2\log n} - \pi_F\|_{\TV} = o(1)\ .$$
Hence,
$$\sum_{F \in \cF_k} \nu_{\alg}(F)\|P^{Tn^2\log n}_F - \pi_F\|_{\TV} \leq
\nu_{\alg}(\cF_{k,\typ}^c)+\sup_{F\in\cF_{k,\typ}}\|P_F^{Tn^2\log n}-\pi_F\|_{\TV}=o(1)\ .$$
and implies both that $\|\pi_{\alg} - \pi_{p,k}\|_{\TV} = o(1)$
and that the Glauber chain has mixing time $O(n^2 \log n)$ over typical $F$.

Since each step of the Glauber dynamics can be implemented in linear time, this means that Algorithm \ref{alg:cap} is a $O(n^3\log n)$-time algorithm, completing the proof.
\end{proof}

\bibliographystyle{abbrv}
\bibliography{refs}

\end{document}